\documentclass[11pt,a4paper]{amsart}
\usepackage{amsmath}
\usepackage{amsthm}
\usepackage{amssymb}
\usepackage{amscd}
\usepackage{amsfonts}
\usepackage{amsbsy}
\usepackage{graphicx}
\usepackage{color}
\usepackage[numbers,sort&compress]{natbib}
\usepackage{mathrsfs}
\usepackage{latexsym,bm}
\usepackage{epsfig}
\usepackage{calligra}
\usepackage[T1]{fontenc}
\usepackage{slashbox,multirow}

\allowdisplaybreaks[4]
\newtheorem {theorem} {Theorem}
\newtheorem {proposition} [theorem]{Proposition}

\newtheorem {lemma}  [theorem]{Lemma}

\newtheorem {remark} [theorem]{Remark}

\newtheorem {definition} [theorem]{Definition}
\newtheorem {proposition and conjecture} [theorem]{Proposition and Conjecture}

\begin{document}

\title[A new Chebyshev criterion and its application]
{A new Chebyshev criterion and its application to planar differential systems}

\author[J. Huang, H. Liang and X. Zhang]
{Jianfeng Huang$^1$, Haihua Liang$^2$, Xiang Zhang$^3$}

\address{$^1$  Department of Mathematics,\ Jinan University,\
Guangzhou\ 510632,\ P.R.\ China} \email{thuangjf@jnu.edu.cn}

\address{$^2$
School of Mathematics and Systems Science,\ Guangdong Polytechnic
Normal University,\ Guangzhou\ 510665,\ P.R.\ China}
\email{1215187342@qq.com}

\address{$^3$
School of Mathematical Sciences and MOE-LSC,\ Shanghai Jiao Tong University,\ Shanghai\ 200240,\ P.R.\ China}
\email{xzhang@sjtu.edu.cn}

\subjclass[2010]{Primary 34C07. Secondary 37C05. Tertiary 35F60}
\keywords{Chebyshev systems; Smooth and piecewise smooth planar differential systems; Separation radials; Melnikov functions; Limit cycles. }

\maketitle

\begin{abstract}
This paper establishes a new Chebyshev criterion for some family of integrals.
By virtue of this criterion we obtain several new Chebyshev families.
With the help of these new families we can answer the conjecture posed by Gasull et al in 2015.
Their applications to other two planar differential systems also show that our approach is simpler and in a unified way to handle many kinds of planar differential systems for estimating the number of limit cycles bifurcating from period annulus.
\end{abstract}


\section{Introduction and statements of main results}
Let $ f_0,f_1,\cdots,f_m $ be smooth functions defined on an interval $E$. We say that the order set $ \{f_0,f_1,\cdots,f_m\} $ is a {\em Chebyshev system} (T-system) on $E$, if any non-trivial linear combination $\lambda_0 f_0+\lambda_1 f_1+\cdots+\lambda_{m}f_{m}$ has at most $ m $ isolated zeros on the interval. Moreover, if this upper bound is sharp taking into account multiplicities, the order set is called an {\em extended Chebyshev system} (ET-system).

If $ \{f_0,f_1,\cdots,f_k\} $ is a T-system (resp. ET-system) for each $ k=0,1,\cdots,m $,  the order set  $ \{f_0,f_1,\cdots,f_m\} $ is called  a {\em complete Chebyshev system} (CT-system) (resp. {\em extended complete Chebyshev system} (ECT-system)) on $E$.

The Chebyshev systems have been applied to a number of mathematical problems, such as the weaken Hilbert 16th problem, critical period problem and so on in the past decades. They were initially used in approximation theory, in the studies of spline functions and boundary value problems, and in the theory of fine moment, see e.g. \cite{taubes21.7,taubes21.1} for more results on this field. They were also utilized in the theory of differential systems for studying versal unfolding of the singularities, see e.g. \cite{taubes21.13,taubes21.18} and their cited references, and for estimating the number of isolated periodic orbits, i.e. limit cycles, bifurcating from period annuli. More concretely, in the latter this theory has been employing to obtain the upper bound for the number of zeros of the Melnikov functions, see e.g. \cite{taubes21.18,taubes17,taubes18,taubes28,taubes2,taubes12,taubes11,taubes15,taubes16,taubes8,taubes19,taubes23,taubes Liu1,taubes Liu2,taubes Wang-Xiao-Han,taubes Huang-Liu-Wang} and the references therein. Actually these studies provide lower bounds for the so called weaken Hilbert 16th problem, see e.g.  \cite{taubes21.2,taubes21.20,taubes21.18} for detail description. The theory of ECT-systems has also been applied to the investigation of the critical periods of the period function defined on period annuli of potential systems, see e.g.  \cite{taubes21.12,taubes Vill-Zhang,taubes Grau-Vill}.

There are some well-known classical ECT-systems. For instance,
\[\begin{array}{cl}
 \{1,\log x,x,x\log x,x^2,x^2\log x,\cdots,x^m,x^m\log x\} & \text{ on } \mathbb R^+:=\{x\in\mathbb R:\ x>0\},\\
 \{1,\cos\theta,\cos2\theta,\cdots,\cos m\theta\} &\text{ on } (0,\pi),\\
 \{(x+a_0)^{-1},(x+a_1)^{-1},\cdots,(x+a_m)^{-1}\} &\text{ on } (-\min\{a_i\},+\infty),
\end{array}
\]
etc. However, there are also many non-trivial families that appear in the works mentioned above. Some of them are defined by a series of integrals, such as the Abelian integrals. Usually, it is not easy to verify the Chebyshev property for such families of functions because of their complicated or unknown explicit expressions of their elements and the cumbersome calculations. For this reason, the problem of efficiently determining the Chebyshev property for a given family of functions, is of wide interest. Here we summarize some related main results on determination of Chebyshev property for several families of functions.

 Gasull et al \cite{taubes12}  in 2002 studied the Chebyshev property for the elliptic
integrals by using the argument principle. Gasull et al \cite{taubes2} in 2012 considered a family of integrals of the form
\begin{align}\label{eq18}
\begin{split}
&\left\{\int_{E}\frac{1}{(1-yg(t))^{\alpha}}dt,\int_{E}\frac{g(t)}{(1-yg(t))^{\alpha}}dt,\cdots,\int_{E}\frac{g^m(t)}{(1-yg(t))^{\alpha}}dt\right\},
\ \ g(t)\not\equiv0,
\end{split}
\end{align}
and characterized their Chebyshev property based on the Gram determinants for a suitable set of functions. Gasull et al \cite{taubes18} in 2012 worked up with another family
\begin{align}\label{eq43}
  \left\{1,y,\cdots,y^{m_0}\right\}\bigcup\left(\bigcup^{n}_{i=1}\left\{(y+a_i)^{\beta},y(y+a_i)^{\beta},\cdots,y^{m}(y+a_i)^{\beta}\right\}\right),
\end{align}
which originates from some kinds of Abelian integrals.
By virtue of the Derivation-Division algorithm they proved that \eqref{eq43} is an ET-system when $\beta\not\in\mathbb Z$ and $a_1,\cdots,a_n$ are different real numbers.
  Grau et al \cite{taubes11} in 2011 provided a Chebyshev criterion for a family of Abelian integrals, via Chebyshev properties of the functions in the integrands.
Cen et al \cite{taubes28} in 2020 obtained  another Chebyshev criterion for a similar family according to asymptotic expansions of the  Wronskians of the family of functions. For more relevant works see e.g. Gavrilov and Iliev \cite{taubes23} in 2003 and Gasull et al \cite{GGM20} in 2020, and their references therein.
\vskip 0.3cm

In this work we are concerned with the Chebyshev property for a family of integrals over some different sets. Let $U$ and $E_0,\cdots,E_n$ be $n+2$ open intervals in $\mathbb R$. Assume that $f_{i,j}=f_{i,j}(t)$ and $G_{i}=G_{i}(t,y)$ are $C^{\omega}(E_i)$ and $C^{\omega}(E_i\times U)$ functions, respectively, where $i=0,\cdots,n$ and $j=0,\cdots,m_i$.
Assume additionally that for each $l=0,\cdots,n+\sum_{i=0}^{n}m_i$, the function $f_{ij}\partial^{l}_{y}G_{i}$ is absolutely integrable in the variable $t$ on $E_i$, and the integral $\int_{E_i}f_{ij}\partial^{l}_{y}G_{i}dt$ is uniformly convergent on any fixed closed subinterval of $U$ if it is an improper integral.
We study the following family of functions
\begin{align}\label{eq1}
\begin{split}
\mathcal{F}&=\bigcup^{n}_{i=0}\{I_{i,0},I_{i,1},\cdots,I_{i,m_i}\}
\end{split}
\end{align}
where
\begin{align}\label{eq2}
I_{i,j}(y):=\int_{E_i}f_{i,j}(t)G_{i}(t,y)dt.
\end{align}
We remark that by assumption,
\begin{align*}
  I^{(l)}_{i,j}(y)=\int_{E_i}f_{i,j}(t)\partial^{l}_{y}G_{i}(t,y)dt \quad \text{ for } \ \ l=0,\cdots,n+\sum_{i=0}^{n}m_i.
\end{align*}

In order to state our results, we introduce other two notions that are closely related to the theory of Chebyshev systems (see Section 2 for more details).
\begin{definition}
  Let $f_0,\cdots,f_{k}$ be analytic functions on an open set $E\subset\mathbb R$.
  The continuous Wronskian of $\{f_0,\cdots,f_{k}\}$ at $t\in E$ is
 \begin{align*}
   W[f_0,\cdots,f_k](t)=\det(f^{(i)}_j(t);\ 0\leq i,j\leq k)
   =\left|
  \begin{array}{ccccccc}
  f_0(t)             &\cdots   &f_{k}(t)\\
  f'_0(t)            &\cdots   &f'_{k}(t)   \\
  \vdots            &\ddots   &\vdots                 \\
 f^{(k)}_0(t)       &\cdots   & f^{(k)}_{k}(t)    \\
  \end{array}
  \right|.
 \end{align*}
  The discrete Wronskian of $\{f_0,\cdots,f_{k}\}$ at $(t_0,\cdots,t_{k})\in E^{k+1}$ is
 \begin{align*}
  D[f_0,\cdots,f_k;t_0,\cdots,t_k]=\det(f_j(t_i);\ 0\leq i,j\leq k)
  =\left|
  \begin{array}{ccccccc}
  f_0(t_0)             &\cdots   &f_{k}(t_0)\\
  f_0(t_1)            &\cdots   &f_{k}(t_1)   \\
  \vdots            &\ddots   &\vdots                 \\
 f_0(t_{k})       &\cdots   & f_{k}(t_{k})    \\
  \end{array}
  \right|.
\end{align*}
\end{definition}

The next is our first main result, which provides a new criterion for determining the Chebyshev property of the family $\mathcal F$.

\begin{theorem}\label{thm1}
The order set of functions, $\mathcal{F}$, is an ECT-system on $U$ if
\begin{itemize}
\item[{\rm(H)}] for each $s\in\{0,\cdots,n\}$, $(k_0,\cdots,k_s)\in \{0,\cdots,m_0\}\times\cdots\times\{0,\cdots,m_s\}$ and $y\in U$,
\begin{align}\label{eq19}
\begin{split}
\prod^{s}_{i=0}&D[f_{i,0},\cdots,f_{i,k_i};t_{i,0},\cdots,t_{i,k_i}]\\
&\cdot W[G_0(t_{0,0},\cdot),\cdots,G_0(t_{0,k_0},\cdot),\cdots,G_s(t_{s,0},\cdot),\cdots,G_s(t_{s,k_s},\cdot)](y)
\end{split}
\end{align}
does not vanish identically and does not change sign for $(t_{0,0},\cdots,t_{0,k_0},\cdots,t_{s,0},$ $\cdots,t_{s,k_s})\in E_0^{k_0+1}\times\cdots\times E_s^{k_s+1}$.
\end{itemize}
\end{theorem}

It will be seen later that, Theorem \ref{thm1} in fact implies that the Chebyshev property of the integrals in \eqref{eq1} is determined by the Chebyshev properties of the functions in the integrands of $I_{i,j}$'s in \eqref{eq2}. The key idea to obtain this result comes from in some sense ``commutativity'' of integration and determinant (see Section 2), which is inspired by the works in \cite{taubes11}.

Denote by $\mathbb Z^{\pm}_0=\mathbb Z^{\pm}\cup\{0\}$. As mentioned above, Gasull et al \cite{taubes2} proved that the family \eqref{eq18} is an ECT-system
on a connected component of the set $\{y\in\mathbb R: (1-yg(t))|_{t\in E}>0\}$ when either
$\alpha\in\mathbb R\backslash(\mathbb Z^-_0)$ or $\alpha\in\mathbb Z^-_0$ with $\alpha\leq -m$.
This family is actually the family $\mathcal F$ with $n=0$, $G_0=(1-yg(t))^{-\alpha}$ and $f_{0,j}=g^{j}$.
Then applying Theorem \ref{thm1} to this special family will yield the same result as that in \cite{taubes2} in a very simple way, see Section 3 for details.


We stress that in the literature there are many other families that can be written in form of the family $\mathcal F$, especially in the studies of limit cycle bifurcations of planar differential systems, not only in the smooth case but also in the piecewise smooth case, which will be illustrated later on in this paper. There we will also show that in some previous published works the corresponding functions $G_i$ are the restrictions of some given function $G$ on $E_i$. For such kind of families we specialize Theorem \ref{thm1} and get the following result, which provides easily verified conditions on determining their Chebyshev property.
\begin{theorem}
  \label{thm2}
Suppose that $E_0,\cdots, E_n$ are non-intersecting intervals, and $E$ is an open interval that contains $\bigcup^n_{i=0} E_i$ in its interior.
Suppose that there exists a $C^{\omega}(E\times U)$ function $G=G(t,y)$ such that $G_i=G|_{E_i}$ for $i=0,\cdots,n$. Then the order set of functions, $\mathcal{F}$ in \eqref{eq1}, is an ECT-system on $U$ if the following hypotheses are satisfied:
\begin{itemize}
  \item [(H.1)] For each $i\in\{0,\cdots,n\}$, the order set of functions $\{f_{i,0},\cdots,f_{i,m_i}\}$ is a CT-system on $E_i$.
  \item [(H.2)] For each $y\in U$, the order set of functions $\{G,\partial_y G,\cdots,\partial^{M}_y G\}$ is a CT-system in the variable $t$ on $E$, where $M={\rm Card}(\mathcal F)-1$.
\end{itemize}
\end{theorem}

Under Proposition \ref{prop1} we will comment that this theorem and the next Theorem \ref{thm4} and Proposition \ref{prop1} have greatly generalized and improved several previously published nice works.

In what follows we focus on a particular type of family that is considered in Theorem \ref{thm2}. More precisely, assume additionally in the theorem that $G$ and $G_i$ are of the form
\begin{align}\label{eq3}
\begin{split}
   G=\frac{1}{(1-yg(t))^{\alpha}},\ \
  G_i=G|_{E_i\times U},
\end{split}
  \end{align}
where $g\in{\rm C}^{\omega}(E)$, $\alpha\in\mathbb R$, $U\subseteq\{y\in\mathbb R: (1-yg(t))|_{t\in E}>0\}$ and $i=0,\cdots,n$ (clearly the family \eqref{eq18} is of this type).
 Then it is not difficult to verify that the hypothesis (H.2) in Theorem \ref{thm2} holds when $g$ is monotonic and either $\alpha\in\mathbb R\backslash(\mathbb Z^-_0)$ or $\alpha\in\mathbb Z^-_0$ with $\alpha\leq 1-{\rm Card}(\mathcal F)$.
 Thus by Theorem \ref{thm2} we obtain the following result.

\begin{theorem}\label{thm4}
  Suppose that $E_0,\cdots, E_n$ are non-intersecting intervals, and $E$ is an open interval that contains $\bigcup^n_{i=0} E_i$.
Suppose that $G_i$ are of the form \eqref{eq3} with $g$ being monotonic on $E$ and $\alpha\in\big(\mathbb R\backslash\mathbb Z^-_0\big)\bigcup\big(\mathbb Z^-_0\cap(-\infty,1-{\rm Card(\mathcal F)}]\big)$.  Then the order set of functions, $\mathcal{F}$ in \eqref{eq1}, is an ECT-system on $U$, if for each $i\in\{0,\cdots,n\}$ the order set of functions $\{f_{i,0},\cdots,f_{i,m_i}\}$ is a CT-system on $E_i$.
\end{theorem}

We now show that some known results, which were obtained by different approaches previously, can be given in a unified way on account of Theorem \ref{thm4}. Also the theorem allows us to deal with some other problems, and obtain some new families of Chebyshev systems. To illustrate these, 
define
\begin{align}\label{eq4}
\begin{split}
  &\mathcal C^{E}_{k}(y)=\int_{E}\frac{\xi_E(\theta)\cos k\theta}{(1-y\cos\nu\theta)^{\alpha}}d\theta,\ \
  \mathcal S^{E}_{k}(y)=\int_{E}\frac{\xi_E(\theta)\sin k\theta}{(1-y\cos\nu\theta)^{\alpha}}d\theta,
\end{split}
\end{align}
where $\nu\in\mathbb Z^+$, $\alpha\in\mathbb R$ and $\xi_E$ is an analytic non-vanishing function on $E$. We have the next proposition.

\begin{proposition}\label{prop1}
  The following statements hold.
\begin{itemize}
    \item[(i)] Let $\mathcal C^{E}_{k}$, $\mathcal S^{E}_{k}$ be defined as in \eqref{eq4} with $\alpha\in\mathbb R\backslash\mathbb Z^-_0$, and let $a_{0},\cdots,a_{n},b_{0},\cdots,b_{n}\in\{0,1\}$ with $\sum_{i=0}^{n}(a_{i}+b_{i})\neq0$
    and $a_i\geq b_i$ for each $i=0,\cdots,n$. The following order sets of functions are ECT-systems on $(-1,1):$
    \begin{align}\label{eq5}
    \begin{array}{c}
       \bigcup^{n}_{i=0}
        \left\{\mathcal S^{E_i}_{1},
               \mathcal S^{E_i}_{2},
               \cdots,
               \mathcal S^{E_i}_{m_i}
               \right\}, \\
       \bigcup^{n}_{i=0}
        \left\{a_{i}\mathcal C^{E_i}_{0},
              a_{i}\mathcal C^{E_i}_{1}
              \cdots,
              a_{i}\mathcal C^{E_i}_{m_i},
              b_{i}\mathcal S^{E_i}_{m_i},
              b_{i}\mathcal S^{E_i}_{m_i-1},
              \cdots,
              b_{i}\mathcal S^{E_i}_{1}
              \right\}
              \backslash\{0\},
    \end{array}
    \end{align}
    where $E_0,E_1,\cdots,E_n$ are non-intersecting open intervals contained in $(0,\frac{\pi}{\nu})$. And
    \begin{align}\label{eq6}
    \begin{array}{c}
       \bigcup^{n}_{i=0}
        \left\{\mathcal S^{E_i}_{k_{i}+2l_i},
              \mathcal S^{E_i}_{k_{i}+2(l_i+1)},
              \cdots,
             \mathcal S^{E_i}_{k_{i}+2m_i}
              \right\}, \\
       \bigcup^{n}_{i=0}
        \left\{a_{i}\mathcal C^{E_i}_{k_{i}},
              a_{i}\mathcal C^{E_i}_{k_{i}+2}
              \cdots,
               a_{i}\mathcal C^{E_i}_{k_{i}+2m_i},
               b_{i}\mathcal S^{E_i}_{k_{i}+2m_i},
               b_{i}S^{E_i}_{k_{i}+2(m_i-1)},
              \cdots,
               b_{i}S^{E_i}_{k_{i}+2l_i}
              \right\}
              \backslash\{0\},
    \end{array}
    \end{align}
    where $k_i\in\{0,1\}$, $l_i=\frac{1+(-1)^{k_i}}{2}$, and $E_0,E_1,\cdots,E_n$ are non-intersecting open intervals contained in $(0,\frac{\pi}{\nu})$ $($resp. $(0,\frac{\pi}{2})\cup(\frac{\pi}{2},\pi)$$)$ when $\nu\geq2$ $($resp. $\nu=1$$)$.
    \item [(ii)] The following order set of functions is an ECT-system on $(-a_n,+\infty):$
        \begin{align}\label{eq7}
          \left\{1,y,\cdots,y^{m_0}\right\}\bigcup\left(\bigcup^{n}_{i=1}\left\{(y+a_i)^{\beta},y(y+a_i)^{\beta},\cdots,y^{m_i}(y+a_i)^{\beta}\right\}\right),
        \end{align}
        where $a_1>a_2>\cdots>a_n\in\mathbb R$, $m_1\geq m_2\geq \cdots\geq m_n\in\mathbb Z^+_0$, $m_0\in\mathbb Z^+_0\cup\{-1\}$ and $\beta\in(\mathbb R\backslash\mathbb Z^+_0)\cap(-\infty,m_0-m_1+1)$. Here we have used the convention: when $m_0=-1$ the set $\{1,y,\cdots,y^{m_0}\}$ is empty.
  \end{itemize}
\end{proposition}

Recall that Gasull et al \cite{taubes2} in 2012 gave a useful ECT-system by particularizing the family \eqref{eq18} with $g=\sin\theta$, $\alpha\in\mathbb R\backslash\mathbb Z_0^-$ and $E=[0,2\pi]$. This is in fact an equivalent type of the family $\big\{\mathcal C^{(0,\pi)}_{0},\cdots,\mathcal C^{(0,\pi)}_{m}\big\}$ with $\xi_{(0,\pi)}=1$ and $\nu=1$. And it is proved in \cite{taubes18}
 that the family \eqref{eq43} (a particular type of the family \eqref{eq7}) mentioned above is an ET-system under some hypothesis. Our result generalizes and improves in a unified way these previous works.

The families presented in Proposition \ref{prop1} have also appeared in some other papers.
For instance,
 the Chebyshev property of the family $\big\{\mathcal C^{[0,\pi]}_{0},\cdots,\mathcal C^{[0,\pi]}_{m}\big\}$ with $\xi_{[0,\pi]}=1$ and $\nu=1$ had been applied by Huang et al in \cite{taubes8}.
Prohens and Torregrosa \cite{taubes19} in 2014 used the explicit expressions of the families in \eqref{eq5} with $\nu=1$ and $\alpha\in\mathbb Z^+$ to investigate periodic orbits via the second order perturbations of some quadratic isochronous centers.
These explicit expressions in some particular forms has also been utilized in \cite{taubes13,taubes16} for obtaining the lower bounds on the Hilbert number of some planar piecewise smooth differential systems.


Here we illustrate three applications, which show that our criterions are useful and powerful in solving concrete problems on the maximum number of isolated zeros of the first order Melnikov functions, which are derived from some integrable non-Hamiltonian systems with smooth/piecewise smooth perturbations. It is noteworthy that such maximum number bounds the number of limit cycles bifurcating from the period annulus of the unperturbed systems up to first order analysis (see Section 2 and Section 5 for details).


The first one is an example of the planar differential systems with homogeneous nonlinearities. This type of systems is extensively studied in the literature (see for instance \cite{taubes1,taubes3,taubes4,taubes5,taubes6,taubes7,taubes9,taubes10}).
For this kind of systems Gaull et al \cite{taubes4} in 2015 got the next results.

\begin{proposition}[\cite{taubes4}]\label{prop3}
Consider the planar polynomial differential systems with homogeneous nonlinearities of degree $n\geq2$ and a singularity at the origin. Let $\mathcal H(n)$ be the maximum number of limit cycles surrounding the origin that such systems can have. Then the estimate of $\mathcal H(n)$ with respect to the types of the singularity and the parity of $n$ are given in Table \ref{table1}, where the notation ``{\bf ?}'' represents the unknown result.
\end{proposition}
\begin{table}\label{table1}
 \centering
 \renewcommand\arraystretch{1.5}
\begin{tabular}{|l|*{6}{c|}}\hline
 &Node &Saddle &Strong focus &Weak focus & Nilpotent singularity \\ \hline
 $n$ is odd &$\geq\left[\frac{n}{2}\right]+1$ &$\geq\left[\frac{n}{2}\right]+1$ &$\geq\left[\frac{n}{2}\right]+1$ &$\geq\left[\frac{n}{2}\right]$ &$\geq\left[\frac{n}{2}\right]$  \\ \hline
 $n$ is even &=0 &=0 &\bf ? &\bf ? &=0 \\ \hline
\end{tabular}
\vskip0.1cm
\caption{Estimates of $\mathcal H(n)$. }\vspace{-0.7cm}
\end{table}

As shown in the table, they were not able to obtain the estimate of $\mathcal H(n)$ when $n$ is even and the origin is a focus. Instead, they posed the next conjecture:

``{\em The result for the remaining cases could be similar to the ones when $n$ is odd}.''

By applying our criterion for ETC systems to the system
\begin{align}\label{eq9}
\frac{d}{dt}
  \left(
\begin{aligned}
&x\\
&y
\end{aligned}
\right)
 =
  \left(
\begin{aligned}
&-y+\frac{1}{2m-1}x^2(x^2+y^2)^{m-1}+\varepsilon P^H_{2m}(x,y)\\
&x+\frac{1}{2m-1}xy(x^2+y^2)^{m-1}+\varepsilon Q^H_{2m}(x,y)
\end{aligned}
\right),
\end{align}
we provide a positive answer to the conjecture posed by Gasull et al mentioned above, where $P^{H}_{2m}$ and $Q^{H}_{2m}$ are homogeneous polynomials of degree $2m$, $m\in\mathbb Z^+$. Note that this last system is a homogeneous polynomial perturbation of a uniform isochronous center.
 For more details on its proof, see Section 5 and Remark \ref{rem2}.


The second application is on the system
\begin{align}\label{eq11}
\frac{d}{dt}
  \left(
\begin{aligned}
&x\\
&y
\end{aligned}
\right)
 =
  \left(
\begin{aligned}
&-y\prod_{i=1}^{n_1}(x-a_i)\prod_{j=1}^{n_2}(y-b_j)+\varepsilon P_m(x,y)\\
&x\prod_{i=1}^{n_1}(x-a_i)\prod_{j=1}^{n_2}(y-b_j)+\varepsilon Q_m(x,y)
\end{aligned}
\right),
\end{align}
where $P_m$, $Q_m$ are polynomials of degree $m$, and $a_i,b_j$ are non-zero real numbers with $a_i\neq a_j$ and $b_i\neq b_j$ for $i\neq j$,  which is a polynomial perturbation of an integrable system having a center at the origin and a family of vertical and/or horizontal lines of singular points.
Gasull et al \cite{taubes20} in 2012 studied the first order Melnikov function of the system to obtain the number of limit cycles via the Derivation-Division algorithm. Here we reobtain the result in a simple way by using ECT-system \eqref{eq7}.

The third application is on piecewise smooth planar differential systems separated by radials
\begin{align*}
\Sigma:=&\bigcup_{i=0}^{n}\left\{(x,y)=(r\cos\vartheta_i,r\sin\vartheta_i):r\in\mathbb R^+\right\},
\end{align*}
 where $n\geq1$ and $\vartheta_0<\vartheta_1<\cdots<\vartheta_n\in[-\pi,\pi)$. Let $A_i$, $i=0,1,\ldots,n$, be the angular sectors limited by $\vartheta_i$ and $\vartheta_{i+1}$, with $\vartheta_{n+1}=\vartheta_0$ by convention.
For piecewise differential systems separated by radials, there are many published results.  Li and Liu \cite{taubes13} in 2015  worked on system \eqref{eq10} with $n=1$, $\vartheta_0=-\frac{\pi}{2}$ and $\vartheta_1=\frac{\pi}{2}$ for $\Sigma$ (a separation straight line). Cardin and Torregrosa \cite{taubes29} in 2015 concerned the case with radials $\vartheta_0=0$ and $\vartheta_1=\frac{\pi}{2}$. Llibre and Zhang \cite{taubes14} in 2019 studied the case with three rays $\vartheta_0=-\pi$, $\vartheta_1=0$ and $\vartheta_2=\frac{\pi}{2}$. Buzzi et al \cite{taubes15} in 2020 focused on the case with four rays $\vartheta_0=-\pi$, $\vartheta_1=-\frac{\pi}{2}$, $\vartheta_2=0$ and $\vartheta_3=\frac{\pi}{2}$. In this direction there are also some other interesting results, see e.g. \cite{taubes15.3,taubes15.13,taubes15.22} for the models from different real phenomena; and \cite{taubes30} with symmetric rays defined by $\exp((n+1)\vartheta\sqrt{-1})=1$ in $\vartheta$.

Here we consider the system
 \begin{align}\label{eq10}
  \frac{d}{dt}
  \left(
\begin{aligned}
&x\\
&y
\end{aligned}
\right)
 =
 \left(
\begin{aligned}
&-y(1-ax)+\varepsilon P_m(x,y)\\
&x(1-ax)+\varepsilon Q_m(x,y)
\end{aligned}
\right),
\end{align}
where $a\neq0$ and $(P_m,Q_m)|_{A_{i}}=(P_{m,i},Q_{m,i})$ with $P_{m,i}$ and $Q_{m,i}$ being polynomials of degree $m$, $i=0,\cdots,n$.
By definition  $(P_m,Q_m)$ on $\Sigma$ obeys the Filippov convention \cite{taubes24}.
From the point of view in analyzing the structure of Melnikov functions for the system with above mentioned separation radials, our approach not only simplifies the calculation (no explicit expression of the integrals is needed) but also unifies several previous different treatings. For details we refer to Proposition \ref{prop8} and its proofs.

\vskip0.3cm
This paper is organized as follows. Section 2 provides some preliminary results, which are used in Section 3 to prove Theorems \ref{thm1}, \ref{thm2} and \ref{thm4}. Section 4 is devoted to giving the Chebyshev families in Proposition \ref{prop1}, and also shows in Remark \ref{rem1} the necessity of the conditions in this proposition.
The three applications of the new Chebyshev families to the three mentioned systems will be presented in Section 5.
The last part is an appendix. 

\section{Preliminaries}\label{an apprpriate label}
For abbreviation we use the notations
\begin{align*}
  x_0,\cdots,x_{k}:=\bm{x}_{k}
  \ \text{and}\
  x_{i,0},\cdots,x_{i,k}:=\bm{x}_{i,k}.
\end{align*}
Accordingly, we write
\begin{align*}
 &f_{i,0},f_{i,1},\cdots,f_{i,k}:=\bm{f}_{i,k},\\
 &I_{i,0},I_{i,1},\cdots,I_{i,k}:=\bm{I}_{i,k},\\
 &I_{0,0},I_{0,1},\cdots,I_{0,k_0},\cdots,I_{s,0},I_{s,1},\cdots,I_{s,k_s}:=\bm I_{0,k_0},\cdots,\bm{I}_{s,k_s},\\
 & W[f_0,f_1,\cdots,f_{k}](t):=W[\bm f_{k}](t),\\
 &D[f_0,f_1\cdots,f_{k};t_0,t_1,\cdots,t_{k}]:=D[\bm f_{k};\bm{t}_{k}],
\end{align*}
for the families of functions, the continuous and discrete Wronskians, respectively.

\subsection{Chebyshev Systems}
As introduced in Section 1, the theory of Chebyshev systems provides a classical tool to study the number of zeros of the linear combinations of some given functions. Here we introduce some related known results (see e.g. \cite{taubesECT1,taubesECT2,taubes21.7} for more details), and also present some of our results on the Chebyshev property for several new families.

First it is clear from definition that if the order set $\{f_0,f_1,\cdots,f_m\}$ is an ECT-system on $E$, then it is a CT-system on $E$, and all the elements are linearly independent. Moreover, the following criterion  enables us to characterize an ECT/CT-system in terms of Wronskians.


\begin{lemma}[\cite{taubes21.7}]\label{lem2.2}
Let $ f_0,f_1,\cdots,f_m $ be analytic functions on an interval $E$ of $\mathbb{R}$. The following statements hold.
\begin{itemize}
  \item [(i)] The order set $\{f_0,f_1,\cdots,f_m\}$ is a CT-system on $E$ if and only if, for each $ k=0,1,\cdots,m $,
  \begin{align*}
    D[\bm f_{k};\bm t_{k}]\neq0 \text{ for all } \bm t_{k}\in E^{k+1} \text{ such that } t_i\neq t_j \text{ for } i\neq j.
  \end{align*}
  \item [(ii)] The order set $\{f_0,f_1,\cdots,f_m\}$ is an ECT-system on $E$ if and only if, for each $ k=0,1,\cdots,m $,
  \begin{align*}
    W[\bm f_{k}](t)\neq0 \text{ for all } t\in E.
  \end{align*}
\end{itemize}
\end{lemma}


\begin{lemma}[\cite{taubesECT2}]\label{lem2.3}
The order set of functions $\{1,f_0,\cdots,f_{m}\}$ is an ECT-system on an interval $ E $, if and only if $ \{f_{0}', \cdots, f_{m}'\} $ is an ECT-system on $E$.
\end{lemma}

\begin{lemma}[\cite{taubesECT2}]\label{lem2.4}
Let $ \xi $ be a smooth non-vanishing function on an interval $ E $. Then the order set of functions  $ \{\xi f_0, \xi f_1, \cdots , \xi f_m \} $ is an ECT-system on $ E $ if and only if $\{f_0, f_1, \cdots , f_m\} $ is an ECT-system on $E$.
\end{lemma}

\begin{theorem}[\cite{taubesECT1}]\label{thm2.5}
A finite family of linearly independent $($real of complex$)$ analytic functions has non-vanishing Wronskian.
\end{theorem}

\begin{proposition}\label{prop7}
  For any fixed $m\in\mathbb Z^+$, the following order sets of functions are ECT-systems on $(0,\pi):$
  \begin{align}\label{eq21}
    \begin{split}
       &\left\{1,\cos\theta,\cdots,\cos m\theta
              \right\}, \\
       &\left\{\sin\theta,\sin2\theta,\cdots,\sin m\theta
              \right\}, \\
       &\left\{1,\cos\theta,\cdots,\cos m\theta,\sin m\theta,\sin(m-1)\theta,\cdots,\sin\theta
              \right\}.
    \end{split}
    \end{align}
  Moreover, the following order sets of functions are ECT-systems on both $(0,\frac{\pi}{2})$ and $(\frac{\pi}{2},\pi):$
  \begin{align}\label{eq22}
    \begin{split}
       &\left\{\cos k\theta,\cos(k+2)\theta,\cdots,\cos(k+2m)\theta
              \right\}, \\
       &\left\{\sin (k+2l)\theta,\sin(k+2l+2)\theta,\cdots,\sin(k+2m)\theta
              \right\}, \\
       &\left\{\cos k\theta,\cos(k+2)\theta,\cdots,\cos(k+2m)\theta,\right.\\
               &\ \left.\sin(k+2m)\theta,\sin(k+2m-2)\theta,\cdots,\sin(k+2l)\theta
              \right\},
    \end{split}
    \end{align}
    where $k\in\{0,1\}$ and $l=\frac{1+(-1)^k}{2}$.
\end{proposition}

We remark that the Chebyshev properties of some families in Proposition \ref{prop7}, such as $\{1,\cos\theta,\cdots,\cos m\theta\}$, have already been known and shown in the previous works, see e.g. \cite{taubes21.7,taubes18}.
The complete proof of the proposition is a little lengthy but follows directly from the above results and Chebyshev polynomials.
For sake of brevity we arrange it in Appendix A.1.

\subsection{A formula on the integration and determinant}
In this subsection we give a result to present some ``commutativity'' of integration and determinant, and then provide an integral expression for the Wronskian of family $\mathcal F$. These are the key tools to prove our main theorems.
\begin{lemma}\label{lem2.6}
  Let $V_0,V_1,\cdots,V_k$ be $k+1$ intervals in $\mathbb R$. Suppose that $\phi_{0,j},\phi_{1,j},\cdots,\phi_{k,j}$ are absolutely integrable functions on $V_{j}$, where $j=0,1,\cdots,k$. Then
  \begin{align*}
    \left|
  \begin{array}{ccccccc}
  \int_{V_0}\phi_{0,0}(t)dt             &\cdots   &\int_{V_k}\phi_{0,k}(t)dt\\
  \int_{V_0}\phi_{1,0}(t)dt            &\cdots   &\int_{V_k}\phi_{1,k}(t)dt   \\
  \vdots            &\ddots   &\vdots                 \\
 \int_{V_0}\phi_{k,0}(t)dt       &\cdots   & \int_{V_k}\phi_{k,k}(t)dt    \\
  \end{array}
  \right|
  =\int\cdots\int_{V_0\times\cdots\times V_k}
    \left|
  \begin{array}{ccccccc}
  \phi_{0,0}(t_0)             &\cdots   &\phi_{0,k}(t_k)\\
  \phi_{1,0}(t_0)            &\cdots   &\phi_{1,k}(t_k)   \\
  \vdots            &\ddots   &\vdots                 \\
 \phi_{k,0}(t_0)       &\cdots   & \phi_{k,k}(t_k)    \\
  \end{array}
  \right|dt_0\cdots dt_k.
  \end{align*}
\end{lemma}

\begin{proof}
  Let $S_{k+1}$ be the symmetric group formed by all permutations of $k+1$ elements $\{0,1,\ldots,k\}$, and let $\tau$ be the evaluation function on $S_{k+1}$ that evaluates a permutation to be the minimum number of transpositions for passing the permutation to the identiy. Then direct calculations show that
  \begin{align*}
 &    \left|
  \begin{array}{ccccccc}
  \int_{V_0}\phi_{0,0}(t)dt             &\cdots   &\int_{V_k}\phi_{0,k}(t)dt\\
  \int_{V_0}\phi_{1,0}(t)dt            &\cdots   &\int_{V_k}\phi_{1,k}(t)dt   \\
  \vdots            &\ddots   &\vdots                 \\
 \int_{V_0}\phi_{k,0}(t)dt       &\cdots   & \int_{V_k}\phi_{k,k}(t)dt    \\
  \end{array}
  \right|\\
  &\indent=\sum_{\rho\in S_{k+1}}(-1)^{\tau(\rho)}
           \prod_{j=0}^{k}\int_{V_j}\phi_{\rho(j),j}(t_j)dt_j\\
  &\indent=\int\cdots\int_{V_0\times\cdots\times V_k}
           \left(\sum_{\rho\in S_{k+1}}(-1)^{\tau(\rho)}
           \prod_{j=0}^{k}\phi_{\rho(j),j}(t_j)\right)
           dt_0\cdots dt_k\\
  &\indent=\int\cdots\int_{V_0\times\cdots\times V_k}
    \left|
  \begin{array}{ccccccc}
  \phi_{0,0}(t_0)             &\cdots   &\phi_{0,k}(t_k)\\
  \phi_{1,0}(t_0)            &\cdots   &\phi_{1,k}(t_k)   \\
  \vdots            &\ddots   &\vdots                 \\
 \phi_{k,0}(t_0)       &\cdots   & \phi_{k,k}(t_k)    \\
  \end{array}
  \right|dt_0\cdots dt_k.
  \end{align*}
The lemma follows.
\end{proof}



\begin{proposition}\label{prop2}
Let $I_{i,j}$ be defined as in \eqref{eq2}. Then for each $s\in\{0,\cdots,n\}$ and $(k_0,\cdots,k_s)\in \{0,\cdots,m_0\}\times\cdots\times\{0,\cdots,m_s\}$,

\begin{align*}
 W&[\bm I_{0,k_0},\cdots,\bm{I}_{s,k_s}](y)\\
  &=\prod_{i=0}^{s}\frac{1}{(k_i+1)!}
  \int\cdots\int_{E_0^{k_0+1}\times\cdots\times E_s^{k_s+1}}\prod^{s}_{i=0}D[\bm f_{i,k_i};\bm t_{i,k_i}]\\
&\qquad \qquad \indent\cdot W[G_0(t_{0,0},\cdot),\cdots,G_0(t_{0,k_0},\cdot),\cdots,G_s(t_{s,0},\cdot),\cdots,G_s(t_{s,k_s},\cdot)](y)dV,
\end{align*}
where  $dV=dt_{0,0}\cdots dt_{0,k_0}\cdots dt_{s,0}\cdots dt_{s,k_s}$.
\end{proposition}

\begin{proof}
Set $M:=s+\sum_{i=0}^{s}k_i$ and $E^{\bm k_s}:=E_0^{k_0+1}\times\cdots\times E_s^{k_s+1}$. According to Lemma \ref{lem2.6}, we have
 \begin{align}\label{eq12}
 \begin{split}
 &W[\bm I_{0,k_0},\cdots,\bm{I}_{s,k_s}](y)\\
  &=
  \left|
  \begin{array}{ccccccc}
  &\cdots         &\int_{E_i}f_{i,0}(t)G_i(t,y)dt                    &\cdots   &\int_{E_i}f_{i,k_i}(t)G_i(t,y)dt                     &\cdots\\
  &\cdots         &\int_{E_i}f_{i,0}(t)\partial_{y}G_i(t,y)dt          &\cdots   &\int_{E_i}f_{i,k_i}(t)\partial_{y}G_i(t,y)dt           &\cdots\\
  &\cdots         &\vdots            &\ddots   &\vdots                                          &\cdots\\
  &\cdots         &\int_{E_i}f_{i,0}(t)\partial^{M}_{y}G_i(t,y)dt      &\cdots   &\int_{E_i}f_{i,k_i}(t)\partial^{M}_{y}G_i(t,y)dt         &\cdots\\
  \end{array}
  \right|\\
  &=
  \int\cdots\int_{E^{\bm k_s}}
  \left|
  \begin{array}{ccccccc}
  &\cdots         &f_{i,0}(t_{i,0})G_i(t_{i,0},y)                      &\cdots   &f_{i,k_i}(t_{i,k_i})G_i(t_{i,k_i},y)                     &\cdots\\
  &\cdots         &f_{i,0}(t_{i,0})\partial_{y}G_i(t_{i,0},y)          &\cdots   &f_{i,k_i}(t_{i,k_i})\partial_{y}G_i(t_{i,k_i},y)           &\cdots\\
  &\cdots         &\vdots               &\ddots   &\vdots                                          &\cdots\\
  &\cdots         &f_{i,0}(t_{i,0})\partial^{M}_{y}G_i(t_{i,0},y)      &\cdots   &f_{i,k_i}(t_{i,k_i})\partial^{M}_{y}G_i(t_{i,k_i},y)         &\cdots\\ \end{array}
  \right|dV\\
 &=
  \int\cdots\int_{E^{\bm k_s}}
  \prod_{i=0}^{s}\left(\prod_{j=0}^{k_i}f_{i,j}(t_{i,j})\right)\\
  &\qquad\qquad \indent\cdot  W[G_0(t_{0,0},\cdot),\cdots,G_0(t_{0,k_0},\cdot),\cdots,G_s(t_{s,0},\cdot),\cdots,G_s(t_{s,k_s},\cdot)](y)dV,
  \end{split}
\end{align}
where the subscript ``$i$'' in the first and second equalities represents the integers from $0$ to $s$.

Let $S_{k_i+1}$ be the symmetric group formed by all permutations of the $k_i+1$ elements $\{0,1,\ldots,k_i\}$, $i=0,\cdots,s$.
For $i\in\{0,1,\ldots,s\}$, since $t_{ij}\in E_i$, $j=0,1,\ldots,k_i$, in the integration \eqref{eq12} we can arbitrarily use the integrating variables $t_{i,\rho_i(0)}, \ldots, t_{i,\rho_i(k_i)}$, where $\rho_i\in S_{k_i+1}$, $i=0,1,\ldots,s$. Then we can write the equality \eqref{eq12} in
\begin{align}\label{eq13}
 \begin{split}
 &W[\bm I_{0,k_0},\cdots,\bm{I}_{s,k_s}](y)\\
 &=
  \int\cdots\int_{E^{\bm k_s}}
  \prod_{i=0}^{s}\left(\prod_{j=0}^{k_i}f_{i,j}(t_{i,\rho_{i}(j)})\right)\\
  &\indent\cdot  W[G_0(t_{0,\rho_{0}(0)},\cdot),\cdots,G_0(t_{0,\rho_{0}(k_0)},\cdot),\cdots,G_s(t_{s,\rho_{s}(0)},\cdot),\cdots,G_s(t_{s,\rho_{s}(k_s)},\cdot)](y)d\overline V\\
  &=
  \int\cdots\int_{E^{\bm k_s}}
  \prod_{i=0}^{s}(-1)^{\tau(\rho_i)}\left(\prod_{j=0}^{k_i} f_{i,j}(t_{i,\rho_{i}(j)})\right)\\
  &\indent\cdot  W[G_0(t_{0,0},\cdot),\cdots,G_0(t_{0,k_0},\cdot),\cdots,G_s(t_{s,0},\cdot),\cdots,G_s(t_{s,k_s},\cdot)](y)dV,
  \end{split}
\end{align}
where $d\overline V=dt_{0,\rho_{0}(0)}\cdots dt_{0,\rho_{0}(k_0)}\cdots dt_{s,\rho_{s}(0)}\cdots dt_{s,\rho_{s}(k_s)}$, and in the second equality we have used the properties of determinant on $W$.
As a result, we get from \eqref{eq13} that
\begin{align*}
\prod_{i=0}^{s}&(k_i+1)!\cdot W[\bm I_{0,k_0},\cdots,\bm{I}_{s,k_s}](y)\\
 &=\sum_{\substack{\rho_i\in S_{k_i+1},\\ i=0,\cdots,s}}
 \int\cdots\int_{E^{\bm k_s}}
  \prod_{i=0}^{s}(-1)^{\tau(\rho_i)}\left(\prod_{j=0}^{k_i}f_{i,j}(t_{i,\rho_{i}(j)})\right)\\
  &\indent\indent\indent\indent\
  \cdot  W[G_0(t_{0,0},\cdot),\cdots,G_0(t_{0,k_0},\cdot),\cdots,G_s(t_{s,0},\cdot),\cdots,G_s(t_{s,k_s},\cdot)](y)dV\\
  &=
 \int\cdots\int_{E^{\bm k_s}}
  \prod_{i=0}^{s}\left(\sum_{\rho_i\in S_{k_i+1}}(-1)^{\tau(\rho_i)}\prod_{j=0}^{k_i}
   f_{i,j}(t_{i,\rho_{i}(j)})\right)\\
  &\indent
  \cdot  W[G_0(t_{0,0},\cdot),\cdots,G_0(t_{0,k_0},\cdot),\cdots,G_s(t_{s,0},\cdot),\cdots,G_s(t_{s,k_s},\cdot)](y)dV\\
  &=
  \int\cdots\int_{E^{\bm k_s}}
  \prod_{i=0}^{s}D[\bm f_{i,k_i};\bm t_{i,k_i}]\\
  &\indent
  \cdot  W[G_0(t_{0,0},\cdot),\cdots,G_0(t_{0,k_0},\cdot),\cdots,G_s(t_{s,0},\cdot),\cdots,G_s(t_{s,k_s},\cdot)](y)dV.
\end{align*}
This proves the proposition.
\end{proof}

\subsection{First order Melnikov function for a kind of one dimensional non-autonomous differential equation}\label{an
appropriate label}
Assume that $H=H(t,x)$ is an analytic function defined in the $(t,x)$-plane, and it satisfies the next hypotheses:
\begin{itemize}
  \item [$\bullet$] $H$ is $2\pi$-periodic in the variable $t$.
  \item [$\bullet$] There exists an open interval $V\subseteq\mathbb R$, such that for each $h\in V$, $H(t,x)=h$ defines a convex periodic curve with respect to $t$, which contains the origin in its interior of the region limited by it.
  \item [$\bullet$] $H_x(t,x)\neq0$ in the region $R=\{(t,x):H(t,x)=h\in V\}$.
\end{itemize}
Assume that $L_i(t,x)$ for $i=0,1,\cdots,k-1$ are analytic functions defined in the region $R$ and are $2\pi$-periodic in $t$.
Let $L(t,x)$ be a function defined in $R$, which is $2\pi$-periodic in $t$ such that
\begin{align*}
L|_{R_i}=L_i|_{R_i},\indent R_i=\{(t,x)\in R: \sigma_i\leq t<\sigma_{i+1}\}, \quad i=0,1,\ldots,k-1,
\end{align*}
where $-\pi=\sigma_0<\sigma_1<\cdots<\sigma_k=\pi$.
We remark that $L$ is smooth when $k=1$, and is piecewise smooth when $k\geq2$.

In this subsection we consider the perturbed differential equation
\begin{align}\label{eq14}
  \frac{dx}{dt}=-\frac{H_t(t,x)}{H_x(t,x)}+\varepsilon L(t,x)+o(\varepsilon).
\end{align}
By the assumption, for $\varepsilon>0$ sufficiently small equation \eqref{eq14} has a unique solution $x_{\varepsilon}(t,\rho)$, which satisfies $x_{\varepsilon}(-\pi,\rho)=\rho\in V$ and is well-defined on $[-\pi,\pi]$ with $H(t,x_{\varepsilon}(t,\rho))\in V$.
Clearly, $H(t,x)$ is a first integral of equation \eqref{eq14}$|_{\varepsilon=0}$. So the region $R$ is a period annulus of the unperturbed equation. From the assumption it follows that all periodic orbits in $R$ transversally cross the lines $\Sigma_0:t=\sigma_0,$ $\Sigma_1:t=\sigma_1,$ $\cdots,$ $\Sigma_{k}:t=\sigma_{k}$. This implies that when $h\in V$, for $ \varepsilon>0 $ sufficiently small the solutions of equation \eqref{eq14} induces a series of Poincar\'{e} maps
\[
\Sigma_0\rightarrow\Sigma_{1}\rightarrow\Sigma_{2}\rightarrow\cdots\rightarrow\Sigma_{k},
\]
defined in a neighborhood of the periodic orbit $H(t,x)=h$.

Recall that a solution $x=x_{\varepsilon}(t,\rho)$ of equation \eqref{eq14} is  $2\pi$-periodic if $x_{\varepsilon}(\pi,\rho)=\rho$. Moreover,
a $2\pi$-periodic solution is  a limit cycle if it is isolated in the set of $2\pi$-periodic solutions.
It is well known that a solution is $2\pi$-periodic (resp. limit cycle) if and only if its initial value  $\rho$ is the zero (resp. isolated zero) of the displacement function
\begin{align*}
\Delta(\rho,\varepsilon):=H(\pi, x_{\varepsilon}(\pi,\rho))-H(-\pi, \rho).
\end{align*}

Note that $ x_{\varepsilon}(t;\rho) $ is smooth with respect to $\varepsilon  $ and $\rho$, because of the smoothness of the composition of two smooth functions. Then $ \Delta(\rho,\varepsilon) $  can be expanded into a power series
\begin{align*}
\Delta(\rho,\varepsilon) = \sum_{i=1}^{+\infty} \varepsilon^i M_i(\rho), \ \text{where} \  M_i(\rho)= \frac{1}{i!}\partial_{\varepsilon}^{i}\left( \Delta(\rho,\varepsilon) \right)|_{\varepsilon=0}.
\end{align*}
For $i\in\mathbb N$, the function $ M_i(\rho) $ is called the $ i $-th  order Melnikov function of equation \eqref{eq14}. We remark that when $ M_1(\rho)=\cdots=M_{k-1}(\rho)\equiv0  $ and $ M_{k}(\rho) \not\equiv 0  $, the limit cycles of equation \eqref{eq14} bifurcating from the period solutions in $R$ are given by the isolated zeros of $ M_{k} $.

The approaches of getting the Melnikov functions of a piecewise smooth equation/system, which can be traced back to the Poincar\'{e}-Pontryagin-Melnikov theory, have been studied in many previous works, see for instance \cite{taubes29,taubes19,taubes Liu-Han,taubes Wei} and the references therein.
For the readers' convenience we present here the expression of the first order Melnikov function of equation \eqref{eq14}.

\begin{lemma}\label{lem2.8}
  The first order Melnikov function of equation \eqref{eq14} is given by
  \begin{align*}
  M_1(\rho)=\sum_{i=0}^{k-1}\int_{\sigma_i}^{\sigma_{i+1}}\big(H_x(t,x)L_i(t,x)\big)\big|_{x=x_0(t,\rho)}dt.
  \end{align*}
\end{lemma}
\begin{proof}
  By assumption and a direct calculation,
  \begin{align*}
    M_1(\rho)
    &=
    \partial_{\varepsilon}\Delta(\rho,\varepsilon)|_{\varepsilon=0}\\
    &=
    \left.\partial_{\varepsilon}
    \left(\sum^{k-1}_{i=0}\int^{\sigma_{i+1}}_{\sigma_i}dH(t, x_{\varepsilon}(t,\rho))\right)\right|_{\varepsilon=0}\\
    &=
    \sum^{k-1}_{i=0}\left.\partial_{\varepsilon}
    \left(\int^{\sigma_{i+1}}_{\sigma_i}\big(\varepsilon H_x(t, x_{\varepsilon}(t,\rho))L(t, x_{\varepsilon}(t,\rho))\big)dt\right)\right|_{\varepsilon=0}\\
    &=
    \sum^{k-1}_{i=0}
    \left(\int^{\sigma_{i+1}}_{\sigma_i} \big(H_x(t, x_{0}(t,\rho))L_i(t, x_{0}(t,\rho))\big)dt\right)\\
    &=
    \sum_{i=0}^{k-1}\int_{\sigma_i}^{\sigma_{i+1}}\big(H_x(t,x)L_i(t,x)\big)\big|_{x=x_0(t,\rho)}dt.
  \end{align*}
The assertion is verified.
\end{proof}

\section{Proof of the main results}\label{an appropriate label}

\begin{proof}[Proof of Theorem \ref{thm1}]
  For each $s\in\{0,\cdots,n\}$ and $k_s\in \{0,\cdots,m_s\}$, we get from Hypothesis (H) and Proposition \ref{prop2} that
  \begin{align*}
 W[\bm I_{0,m_0},\cdots,\bm{I}_{s-1,m_{s-1}},\bm{I}_{s,k_s}](y)
  \neq0 \text{ for $y\in U$}.
\end{align*}
This together with Lemma \ref{lem2.2} implies that the order set of functions, $\mathcal F$, is an ECT-system.
\end{proof}

As a consequence of Theorem \ref{thm1}, one reobtains the Chebyshev property of family \eqref{eq18}, which has been verified by Gasull et al in  \cite[Theorem A] {taubes2}. Here our calculation is very easy, as you can check.
Denote by $\mathcal V(\bm t_{k})$ the $(k+1)$-th order Vandermonde determinant for $\bm t_{k}\in\mathbb R^{k+1}$, i.e.,
\begin{align*}
  \mathcal V(\bm t_{k})
  =
  D[1,t,\cdots,t^{k};t_0,t_1,\cdots,t_{k}].
  \end{align*}
Recall that family \eqref{eq18} is the family $\mathcal F$ particularized with $n=0$, $G_0=(1-yg(t))^{-\alpha}$ and $f_{0,j}=g^{j}$. Then for each $k_0=0,1,\cdots,m_0$, we have by a direct calculation that
\begin{align*}
  D[f_{0,0},\cdots,f_{0,k_0};t_{0,0},\cdots,t_{0,k_0}]
  &=
  \mathcal V[g(t_{0,0}),\cdots,g(t_{0,k_0})]\\
  &=
  \prod_{0\leq j<i\leq k_0}\left(g(t_{0,i})-g(t_{0,j})\right),
\end{align*}
and
\begin{align*}
  W&[G_0(t_{0,0},\cdot),\cdots,G_0(t_{0,k_0},\cdot)](y)\\
  &=
  \frac{\prod_{i=0}^{k_0-1}(\alpha+i)^{k_0-i}}{\prod^{k_0}_{i=0}(1-yg(t_{0,i}))^{\alpha}}\cdot
  \mathcal V\left[\frac{g(t_{0,0})}{1-yg(t_{0,0})},\cdots,\frac{g(t_{0,k_0})}{1-yg(t_{0,k_0})}\right]\\
  &=
  \frac{\prod_{i=0}^{k_0-1}(\alpha+i)^{k_0-i}}{\prod^{k_0}_{i=0}(1-yg(t_{0,i}))^{\alpha+k_0}}\cdot
  \prod_{0\leq j<i\leq k_0}\left(g(t_{0,i})-g(t_{0,j})\right).
\end{align*}
Hence, when
$\alpha\in\big(\mathbb R\backslash\mathbb Z^-_0\big)\bigcup\big(\mathbb Z^-_0\cap(-\infty,-m]\big)$,
\[
D[f_{0,0},\cdots,f_{0,k_0};t_{0,0},\cdots,t_{0,k_0}]
\cdot W[G_0(t_{0,0},\cdot),\cdots,G_0(t_{0,k_0},\cdot)](y)\geq0,
\]
where $y$ is contained in any given connected component $J$ of the set $\{y\in\mathbb R: (1-yg(t))|_{t\in E}>0\}$, and the equality only happens when $t_{0,i}=t_{0,j}$ for some $i\neq j$.  According to Theorem \ref{thm1},
family \eqref{eq18} is an ECT-system on $J$. We supplement that when $\alpha\in\mathbb Z^-_0\cap(-m,0]$ the family is not an ECT-system because all the integrals are polynomials of degree no more than $m-1$.

We now turn to the proofs of the other theorems.

\begin{proof}[Proof of Theorem \ref{thm2}]
  We only need to show that hypotheses (H.1) and (H.2) imply
  hypothesis (H) of Theorem \ref{thm1}. To this aim set $s\in\{0,\cdots,n\}$ and $(k_0,\cdots,k_s)\in \{0,\cdots,m_0\}\times\cdots\times\{0,\cdots,m_s\}$. As before, let
   $S_{k_i+1}$ be the symmetric group formed by all permutations of the $k_i+1$ elements $\{0,1,\ldots,k_i\}$, $i=0,\cdots,s$.
  Denote by
  \begin{itemize}
  \item[$\bullet$] $id_i$ the identity element of $S_{k_i+1}$, i.e., $id_i(j)=j$, $j=0,1,\cdots,k_i$.
  \item[$\bullet$] $E_{i,\rho}=\{\bm t_{i,k_i}\in E_{i}^{k_i+1}:t_{i,\rho(0)}<t_{i,\rho(1)}<\cdots<t_{i,\rho(k_i)}\}$, $\rho\in S_{k_i+1}$.
  \end{itemize}
 For convenience set
\begin{align*}
\omega&(\bm t_{0,k_0},\cdots,\bm t_{s,k_s},y)\\
&:=
 \prod_{i=0}^{s}D[\bm f_{i,k_i};\bm t_{i,k_i}]
  \cdot
 W[G_0(t_{0,0},\cdot),\cdots,G_0(t_{0,k_0},\cdot),\cdots,G_s(t_{s,0},\cdot),\cdots,G_s(t_{s,k_s},\cdot)](y).
 \end{align*}

According to hypothesis (H.1) and Lemma \ref{lem2.2},
  \begin{align*}
    D[\bm f_{i,k_i};\bm t_{i,k_i}]\neq0 \text{ for } \bm t_{i,k_i}\in E_{i,id_i} \text{ and } i=0,\cdots,s.
  \end{align*}
Moreover, since $E_0,\cdots,E_s$ are non-intersecting intervals, we can suppose without loss of generality that $t_0<t_1<\cdots<t_s$ for any $t_i\in E_i$, $i\in\{0,1,\ldots,s\}$. That is to say, for $(\bm t_{0,k_0},\cdots,\bm t_{s,k_s})\in E_{0,id_0}\times\cdots\times E_{s,id_s}$
   we have that
  \begin{align*}
    t_{0,0}<t_{0,1}<\cdots<t_{0,k_0}<t_{1,0}<t_{1,1}<\cdots<t_{1,k_1}<\cdots<t_{s,0}<t_{s,1}<\cdots<t_{s,k_s}.
  \end{align*}
  Thus, hypothesis (H.2) and Lemma \ref{lem2.2} ensure that
  \begin{align*}
  W&[G_0(t_{0,0},\cdot),\cdots,G_0(t_{0,k_0},\cdot),\cdots,G_s(t_{s,0},\cdot),\cdots,G_s(t_{s,k_s},\cdot)](y)\\
  &=
  D[G(\cdot,y),\partial_{y}G(\cdot,y),\cdots,\partial^{K}_{y}G(\cdot,y);\bm t_{0,k_0},\bm t_{1,k_1},\cdots,\bm t_{s,k_s}]
  \neq0,
  \end{align*}
   where $K=s+\sum_{i=0}^{s}k_s$ , $(\bm t_{0,k_0},\cdots,\bm t_{s,k_s})\in E_{0,id_0}\times\cdots\times E_{s,id_s}$ and $y\in U$.
  As a result,
  \begin{align}\label{eq20}
  \begin{split}
  \omega&(\bm t_{0,k_0},\cdots,\bm t_{s,k_s},y)\neq0 \text{ on }  E_{0,id_0}\times\cdots\times E_{s,id_s}\times U.
  \end{split}
  \end{align}

Now for $\rho_i\in S_{k_i+1}$, $i=0,\cdots,s$, using the similar tricks as above one gets that
  \begin{align*}
D&[\bm f_{i,k_i};\bm t_{i,k_i}]\\
    &=
    (-1)^{\tau(\rho_i)}\cdot
    D[f_{i,0},\cdots,f_{i,k_i};t_{i,\rho_i(0)},\cdots,t_{i,\rho_i(k_i)}],\\
W&[G_0(t_{0,0},\cdot),\cdots,G_0(t_{0,k_0},\cdot),\cdots,G_s(t_{s,0},\cdot),\cdots,G_s(t_{s,k_s},\cdot)](y)\\
  &=
  \prod_{i=0}^{s}(-1)^{\tau(\rho_i)}\\
  &\indent \cdot W[G_0(t_{0,\rho_0(0)},\cdot),\cdots,G_0(t_{0,\rho_0(k_0)},\cdot),\cdots,G_s(t_{s,\rho_s(0)},\cdot),\cdots,G_s(t_{s,\rho_s(k_s)},\cdot)](y).
  \end{align*}
These induce
  \begin{align*}
  \omega&(t_{0,0},\cdots,t_{0,k_0},\cdots,t_{s,0},\cdots,t_{s,k_s},y)\\
  &=
  \omega(t_{0,\rho_0(0)},\cdots,t_{0,\rho_0(k_0)},\cdots,t_{s,\rho_s(0)},\cdots,t_{s,\rho_s(k_s)},y).
  \end{align*}
Taking  \eqref{eq20} and the arbitrariness of $\rho_i$ into account, it follows that $\omega$ keeps its sign on
  $\big(\bigcup_{\rho\in S_{k_0+1}}E_{0,\rho}\big)\times\cdots\times \big(\bigcup_{\rho\in S_{k_s+1}}E_{s,\rho}\big)\times U$.
  Moreover,  for each $i\in\{0,1,\cdots,s\}$,
  $E_{i}^{k_i+1}\backslash\big(\bigcup_{\rho\in S_{k_i+1}}E_{i,\rho}\big)$
  is a zero Lebesgue measure subset of $E_{i}^{k_i+1}$. Hence by continuity, for each $y\in U$, $\omega(\cdot,y)$ does not vanish identically and does not change its sign on $E_{0}^{k_0+1}\times\cdots\times E_{s}^{k_s+1}$.
  The hypothesis (H) of Theorem \ref{thm1} is verified and so is our assertion.
\end{proof}

\begin{proof}[Proof of Theorem \ref{thm4}]
By the assumption of the theorem and Theorem \ref{thm2}, the assertion of the theorem is true once the family $\mathcal F$ satisfies the hypothesis (H.2) of Theorem \ref{thm2}.

Let $G$ be defined as in \eqref{eq3}. For each $K\in\{0,1,\cdots,{\rm Card(\mathcal F)}-1\}$ and $\bm t_K\in E^{K+1}$, we have
\begin{align*}
D&[G(\cdot,y),\partial_{y}G(\cdot,y),\cdots,\partial^{K}_{y}G(\cdot,y);\bm t_K]\\
  &=
  \frac{\prod^{K-1}_{i=0}(\alpha+i)^{K-i}}{\prod^{K}_{i=0}(1-yg(t_{i}))^{\alpha}}\cdot
  \mathcal V\left[\frac{g(t_{0})}{1-yg(t_{0})},\cdots,\frac{g(t_{K})}{1-yg(t_{K})}\right]\\
  &=
  \frac{\prod^{K-1}_{i=0}(\alpha+i)^{K-i}}{\prod^{K}_{i=0}(1-yg(t_{i}))^{\alpha+K}}\cdot
  \prod_{0\leq j<i\leq K}\left(g(t_{i})-g(t_{j})\right),
\end{align*}
where the notation ``$\mathcal V$'' represents the Vandermonde determinant, as we have used previously.
Since $g$ is monotonic on $E$ and $\alpha\in\big(\mathbb R\backslash\mathbb Z^-_0\big)\bigcup\big(\mathbb Z^-_0\cap(-\infty,1-{\rm Card(\mathcal F)}]\big)$, we know that $g(t_{i})-g(t_{j})\neq0$ when $t_i\neq t_j$, and that $\alpha+i\neq0$ when $i\in\{0,1,\cdots,K-1\}$. Hence,
$$D[G(\cdot,y),\partial_{y}G(\cdot,y),\cdots,\partial^{K}_{y}G(\cdot,y);\bm t_K]\neq0$$
for all $\bm t_K\in E^{K+1}$ satisfying $t_i\neq t_j$ for $i\neq j$. Applying Lemma \ref{lem2.2}, the hypothesis (H.2) of Theorem \ref{thm2} holds for $\mathcal{F}$.

As a result, our assertion is verified by using  Theorem \ref{thm2}.
\end{proof}

\section{Proof of Proposition \ref{prop1}}\label{an appropriate label}

This section is to study the families in Proposition \ref{prop1}. Once we complete the proof of Proposition \ref{prop1}, we will obtain some new families of Chebyshev systems. Now we prove this proposition.

(i) For the sake of shortness we only prove the assertion for the first order set in \eqref{eq5}, because the others of statement (i) follow in exactly the same way from the arguments used in the proof of the first order set.

  Take $g(\theta)=\cos\nu\theta$.
  It is easy to see that the first order set in \eqref{eq5} is the family $\mathcal F$ with
  \[
  \begin{array}{l}
 f_{i,j}=\xi_{E_i}(\theta)\sin(j+1)\theta,\\
 G_i=\left.\frac{1}{(1-yg(\theta))^{\alpha}}\right|_{(\theta,y)\in E_i\times(-1,1)},
 \end{array} \quad
 i=0,\cdots,n,\ j=0,\cdots,m_i-1.
  \]
By the assumption of the proposition we know that $\xi_{E_i}\neq0$ and $g$ is monotonic on $(0,\frac{\pi}{\nu})\subseteq(0,\pi)$. Hence according to Proposition \ref{prop7}, Lemma \ref{lem2.4} and Theorem \ref{thm4}, the first order set in \eqref{eq5} is an ECT-system on $(-1,1)$.

(ii) Set
\begin{align*}
\mathcal A_{k,a,\alpha}
:=
\left\{(y+a)^{\alpha},y(y+a)^{\alpha},\cdots,y^{k}(y+a)^{\alpha}\right\},\indent k\in\mathbb Z_0^+,\ a,\alpha\in\mathbb R.
\end{align*}
From Lemma \ref{lem2.3}, some direct calculations show that the assertion of this statement is equivalent to that the order set of functions
\begin{align}\label{eq23}
\bigcup^{n}_{i=1}
\mathcal A_{m_i,a_i,\beta-m_0-1}
\end{align}
is an ECT-system on $(-a_n,+\infty)$. Next we only need to consider the case $a_n\geq0$, for otherwise we can take a translation of the variable $y$.

Now let $a_0\in(a_1,+\infty]$ to be determined. We consider a kind of family $\mathcal F$ where
\[
\begin{array}{ll}
 U=(-a_n,+\infty),\ E_i=(a_{i+1},a_{i}), & \ i=0,\cdots,n-1,\\
 f_{i,j}=t^{j-\alpha},\ \ \ \qquad  \alpha=2+m_0-\beta, &\  j=0,\cdots,m_{i+1}  \\
  G_i=\left.\frac{1}{(1-y(-t)^{-1})^{\alpha}}\right|_{(t,y)\in E_i\times U}=\left.\frac{t^{\alpha}}{(t+y)^{\alpha}}\right|_{(t,y)\in E_i\times U}, & \ i=0,\cdots,n-1.
\end{array}
\]
By the assumption of the statement, one has $\alpha>1+m_1\geq\cdots\geq1+m_n$. Thus for $j\in\{0,1,\cdots,m_{i+1}\}$,
\begin{align*}
I_{i,j}(y)
&=
\int_{a_{i+1}}^{a_i}\frac{t^j}{(t+y)^{\alpha}}dt=
\sum^{j}_{l=0}\binom{j}{l}(-y)^{j-l}\int^{a_{i}}_{a_{i+1}}\frac{1}{(t+y)^{\alpha-l}}dt\\
&=
\left(\sum^{j}_{l=0}\binom{j}{l}\frac{(-y)^{j-l}(y+a_{i})^{l}}{1+l-\alpha}\right)(y+a_{i})^{1-\alpha}\\
&\indent
\qquad \quad -\left(\sum^{j}_{l=0}\binom{j}{l}\frac{(-y)^{j-l}(y+a_{i+1})^{l}}{1+l-\alpha}\right)(y+a_{i+1})^{1-\alpha}\\
&\in
\left\langle\mathcal A_{j,a_i,1-\alpha}\cup\mathcal A_{j,a_{i+1},1-\alpha}\right\rangle,
\end{align*}
where $\langle \cdot \rangle$ denotes the linear span by its elements.
Note that
\begin{align*}
\lim_{a_0\rightarrow+\infty}\left(\sum^{j}_{l=0}\binom{j}{l}\frac{(-y)^{j-l}(y+a_0)^{l}}{1+l-\alpha}\right)(y+a_0)^{1-\alpha}
=0.
\end{align*}
Hence, taking $a_0=+\infty$, we have that $I_{0,j}\in\langle\mathcal A_{j,a_1,1-\alpha}\rangle$. Consequently, for each $s\in\{1,\cdots,n\}$ and $(j_1,\cdots,j_s)\in\{0,\cdots,m_1\}\times\cdots\times\{0,\cdots,m_s\}$ with $j_1\geq j_2\geq\cdots\geq j_s$,
\begin{align*}
&\left\langle\bigcup^{s-1}_{i=0}\{I_{i,0}(y),I_{i,1}(y),\cdots,I_{i,j_{i+1}}(y)\}\right\rangle
\subseteq
\left\langle\bigcup^{s}_{i=1}\mathcal A_{j_i,a_i,1-\alpha}\right\rangle.
\end{align*}

On the other hand, it follows from Theorem \ref{thm4} that the family $\mathcal F$ here is an ECT-system on $(-a_n,+\infty)$. Thus,
\begin{align*}
  s+\sum_{i=1}^{s}j_i
  &=\dim\left(\left\langle\bigcup^{s-1}_{i=0}\{I_{i,0}(y),I_{i,1}(y),\cdots,I_{i,j_{i+1}}(y)\}\right\rangle\right)\\
&\leq
\dim\left(\left\langle\bigcup^{s}_{i=1}\mathcal A_{j_i,a_i,1-\alpha}\right\rangle\right)
\leq s+\sum_{i=1}^{s}j_i.
\end{align*}
This means that for each $s\in\{1,\cdots,n\}$ and $(j_1,\cdots,j_s)\in\{0,\cdots,m_1\}\times\cdots\times\{0,\cdots,m_s\}$ with $j_1\geq j_2\geq\cdots\geq j_s$, we actually have
\begin{align*}
&\left\langle\bigcup^{s}_{i=1}\mathcal A_{j_i,a_i,1-\alpha}\right\rangle
=
\left\langle\bigcup^{s-1}_{i=0}\{I_{i,0}(y),I_{i,1}(y),\cdots,I_{i,j_{i+1}}(y)\}\right\rangle.
\end{align*}
Consequently, the maximum number of isolated zeros of the non-trivial linear combinations of the elements in  $\bigcup^{s}_{i=1}\mathcal A_{j_i,a_i,1-\alpha}$, is equal to the one in  $\bigcup^{s-1}_{i=0}\{I_{i,0}(y),I_{i,1}(y),\cdots,$ $I_{i,j_{i+1}}(y)\}$.
By definition and the Chebyshev property of $\mathcal F$, the order set $\bigcup^{n}_{i=1}\mathcal A_{m_i,a_i,1-\alpha}$ (i.e., $\bigcup^{n}_{i=1}
\mathcal A_{m_i,a_i,\beta-m_0-1}$) is an ECT-system on $(-a_n,+\infty)$.

It completes the proof of  statement (ii) and so of the proposition.
\qed

\begin{remark}\label{rem1}
{\rm According to  Proposition \ref{prop1}, we have the next comments.
\begin{itemize}
\item  We do not consider the case when $\alpha\in\mathbb Z^-_0$ in statement (i) of Proposition \ref{prop1} because $\mathcal C^E_k$ and $\mathcal S^E_k$ are reduced to polynomials in this case.

\item We give two examples to show the necessity of the condition in statement (ii) of Proposition \ref{prop1}. For the order set of functions
\begin{align*}
&\left\{(y+4)^{\frac{3}{2}},y(y+4)^{\frac{3}{2}},y^2(y+4)^{\frac{3}{2}},y^3(y+4)^{\frac{3}{2}},(y+1)^{\frac{3}{2}}\right\},\\
&\left\{(y+5)^{\frac{5}{2}},y(y+5)^{\frac{5}{2}},y^2(y+5)^{\frac{5}{2}},y^3(y+5)^{\frac{5}{2}},(y+1)^{\frac{5}{2}}\right\},
\end{align*}
one can check that the first (resp. second) family is of the form $\mathcal F$ with $n=2$, $a_1=4,\ a_2=1$, $m_0=-1,\ m_1=3,\ m_2=0$ and $\beta=\frac{3}{2}$ (resp. $n=2$, $a_1=5,\ a_2=1$, $m_0=-1,\ m_1=3,\ m_2=0$ and $\beta=\frac{5}{2}$). Thus in both cases we have $\beta>m_0-m_1+1$, which do not satisfy the condition $\beta\in(\mathbb R\backslash\mathbb Z^+_0)\cap(-\infty,m_0-m_1+1)$ of  Proposition \ref{prop1}.
On the other hand,
it is easy to verify that both of the corresponding fifth order Wronskians of these two families have indefinite signs on $(-1,+\infty)$. Therefore the families are not ECT-systems.

\item There is another example $\{1,y,\sqrt{y+1},y\sqrt{y+1},\sqrt{y},y\sqrt{y},y^2\sqrt{y}\}$, which contradicts with the conditions $a_1>a_2$ and $m_1\geq m_2$ in statement (ii) of Proposition \ref{prop1}.
 In fact, it was proved in \cite{taubes18} that this family is even not an ET-system.
\end{itemize}
}
\end{remark}

\section{Applications to planar differential systems}\label{an appropriate label}

In this section we apply our main results to study limit cycle bifurcation of systems \eqref{eq9}, \eqref{eq11} and \eqref{eq10}, respectively.

\subsection{Smooth planar differential systems}\label{an
appropriate label}
First we point out that system \eqref{eq9}$|_{\varepsilon=0}$ has an isochronous center at the origin, and has the first integral $(x^2+y^2)^{\frac{1}{2}-m}+y(x^2+y^2)^{-\frac{1}{2}}$.

Writing system \eqref{eq9} in polar coordinates yields
\begin{align*}
  &\frac{dr}{dt}
  =
  \frac{\cos\theta}{2m-1}r^{2m}
  +\varepsilon\left(P^H_{2m}(\cos\theta,\sin\theta)\cdot\cos\theta+Q^H_{2m}(\cos\theta,\sin\theta)\cdot\sin\theta\right)r^{2m},\\
  &\frac{d\theta}{dt}
  =
  1+\varepsilon\left(Q^H_{2m}(\cos\theta,\sin\theta)\cdot\cos\theta-P^H_{2m}(\cos\theta,\sin\theta)\cdot\sin\theta\right)r^{2m-1}.
\end{align*}
Then, the Melnikov functions for the system can be obtained by studying the equation
\begin{align}\label{eq26}
  \frac{dr}{d\theta}
  &=
  \frac{\displaystyle\frac{\cos\theta}{2m-1} r^{2m}
  +\varepsilon\left(P^H_{2m}(\cos\theta,\sin\theta)\cdot\cos\theta+Q^H_{2m}(\cos\theta,\sin\theta)\cdot\sin\theta\right)r^{2m}}
  {1+\varepsilon\left(Q^H_{2m}(\cos\theta,\sin\theta)\cdot\cos\theta-P^H_{2m}(\cos\theta,\sin\theta)\cdot\sin\theta\right)r^{2m-1}}.
\end{align}

By taking advantage of the series expression in $\varepsilon$, equation \eqref{eq26} can be written in the form \eqref{eq14} with
\begin{align*}
  &H(\theta,r)=\sin\theta+r^{1-2m}, \\
  &L(\theta,r)=
  \left(P^H_{2m}(\cos\theta,\sin\theta)\cdot\cos\theta+Q^H_{2m}(\cos\theta,\sin\theta)\cdot\sin\theta\right)r^{2m}\\
  &\indent\indent\indent\indent
  -\frac{\left(Q^H_{2m}(\cos\theta,\sin\theta)\cdot\cos\theta-P^H_{2m}(\cos\theta,\sin\theta)\cdot\sin\theta\right)\cos\theta}{2m-1}r^{4m-1},\\
  &\qquad\qquad
  \qquad (\theta,r)\in[-\pi,\pi]\times\mathbb R^+.
\end{align*}
Some easy calculations show that the periodic solution of \eqref{eq26}$|_{\varepsilon=0}$ with the initial value $\rho$ at $\theta=-\pi$ is
$$r_0(\theta,\rho)=\frac{\rho}{\left(1-\rho^{2m-1}\sin\theta\right)^{\frac{1}{2m-1}}},\indent  \rho\in(0,1).$$

Note that $\int_{-\pi}^{\pi}\sin^{i}\theta\cos^j\theta d\theta=0$ when $i,j\in\mathbb Z^+_0$ and $i+j$ is odd. Hence according to Lemma \ref{lem2.8}, the first order Melnikov funcion for equation \eqref{eq26} (also system \eqref{eq9}) is
\begin{align*}
M_1(\rho)
&=
\rho^{2m-1}\int_{-\pi}^{\pi}
\frac{\left(Q^H_{2m}(\cos\theta,\sin\theta)\cdot\cos\theta-P^H_{2m}(\cos\theta,\sin\theta)\cdot\sin\theta\right)\cos\theta}{1-\rho^{2m-1}\sin\theta}d\theta\\
&=
\rho^{2m-1}\int_{-\pi}^{\pi}
\frac{\left(Q^H_{2m}(\sin\theta,\cos\theta)\cdot\sin\theta-P^H_{2m}(\sin\theta,\cos\theta)\cdot\cos\theta\right)\sin\theta}{1-\rho^{2m-1}\cos\theta}d\theta,
\end{align*}
where in the second equality we have used the change of variables $\theta\mapsto\frac{\pi}{2}-\theta$ and the periodicity of sine and cosine functions.
Moreover, using the parities of these trigonometrical functions and the equality $\sin^{2k}\theta=(1-\cos^2\theta)^k$, we can write $M_1$ as
\begin{align*}
M_1(\rho;\bm{\lambda}_m)
&=
\rho^{2m-1}\int_{-\pi}^{\pi}
\frac{\left(\lambda_0+\lambda_1\cos^{2}\theta+\cdots+\lambda_m\cos^{2m}\theta\right)\sin^2\theta}{1-\rho^{2m-1}\cos\theta}d\theta,
\end{align*}
where the coefficients
 $\lambda_0,\cdots,\lambda_m$ are determined by $(P^H_{2m},Q^H_{2m})$ and can be chosen arbitrarily.
If we additionally take the Chebyshev polynomial of the first
kind into account, then there exists a bijective map $T:\mathbb R^{m+1}\rightarrow\mathbb R^{m+1}$ such that $\hat{\bm{\lambda}}_m=T(\bm{\lambda}_m)$ satisfies
\begin{align*}
M_1(\rho;\bm{\lambda}_m)
&=
\rho^{2m-1}\int_{-\pi}^{\pi}
\frac{\left(\hat{\lambda}_0+\hat{\lambda}_1\cos2\theta+\cdots+\hat{\lambda}_m\cos2m\theta\right)\sin^2\theta}{1-\rho^{2m-1}\cos\theta}d\theta\\
&=
2\rho^{2m-1}
\sum_{k=0}^{m}\hat{\lambda}_k\mathcal C^{E}_{2k}(\rho^{2m-1}),
\end{align*}
where $\mathcal C^E_{2k}$'s are defined in \eqref{eq4} with $E=(0,\pi)$, $\xi_E=\sin^2\theta$ and $\alpha=1$.

Now applying the ECT-system $\left\{\mathcal C^E_{0}(y),\mathcal C^E_{1}(y),\cdots,\mathcal C^E_{2m}(y)\right\}$, as stated in statement (i) of Proposition \ref{prop1}, we get that the function $\sum_{k=0}^{m}\hat{\lambda}_k\mathcal C^{E}_{2k}(y)$ has at most $2m$ isolated zeros on $(-1,1)$, counted with multiplicities. Observe that each $\mathcal C^E_{2k}$ is even. Hence, the number of positive isolated zeros of $\sum_{k=0}^{m}\hat{\lambda}_k\mathcal C^{E}_{2k}(y)$ is at most $m$ (counted with multiplicities). On the other hand, the Chebyshev property of the family also ensures the independence of $\mathcal C^{E}_{0},\mathcal C^{E}_{2},\cdots,\mathcal C^{E}_{2m}$. Thus taking Theorem \ref{thm2.5}, Lemma \ref{lem2.2} and the bijectivity of $T$ into account, we know that this upper bound is reachable. Accordingly, the maximum number of isolated zeros of the Melnikov function $M_1$ on $(0,1)$ is $m$, counted with multiplicities (the substitution $y=\rho^{2m-1}$ is a diffeomorphism on $(0,1)$).

\begin{remark}\label{rem2}
{\rm
 The above result in fact means that, there exists planar polynomial differential system with homogeneous nonlinearity of any prescribed even degree $n$, such that the system has $\frac{n}{2}$ limit cycles surrounding the origin (a weak focus). Furthermore, if we add a suitable small extra linear perturbation under which the stability of the origin is changed, there happens a Hopf bifurcation. 
 Thus, we are able to complete Table \ref{table1} with the next result.

 \begin{proposition}\label{p*1}
 In Table \ref{table1} with $n$ even,
   $\mathcal H(n)\geq \frac{n}{2} $ in the case of weak focus, whereas  $\mathcal H(n)\geq  \frac{n}{2}+1$ in the case of  strong focus.
   \end{proposition}

We remark that this proposition provides a positive answer to the conjecture posed in \cite{taubes4}.

}
\end{remark}
\vskip0.3cm

Next we focus on system \eqref{eq11}. In \cite{taubes20} Gasull, L\'azaro and Torregrosa studied the upper bound for the number of limit cycles bifurcating from
the period annulus
\begin{align*}
  \left\{(x,y)\in\mathbb R^2:0<x^2+y^2<\min_{i,j}\{a^2_i,b^2_j\}\right\}
\end{align*}
of system \eqref{eq11}$|_{\varepsilon=0}$.
To this aim they estimate the number of isolated zeros of the first order Melnikov function of the system.
Using the polar coordinates, the problem can be transferred to consider the following perturbed equation
\begin{align*}
  &\frac{dr}{d\theta}
  =
  \varepsilon
  \frac{P_{m}(r\cos\theta,r\sin\theta)\cdot\cos\theta+Q_{m}(r\cos\theta,r\sin\theta)\cdot\sin\theta}
  {\prod_{i=1}^{n_1}(r\cos\theta-a_i)\prod_{j=1}^{n_2}(r\sin\theta-b_j)}
  +o(\varepsilon),
\end{align*}
where $(\theta,r)\in[-\pi,\pi]\times\big(0,\min_{i,j}\{|a_i|,|b_j|\}\big)$.
Clearly, this equation is of the form \eqref{eq14} with $H(\theta,r)=r$, and $r_0(\theta,\rho)=\rho$.
Similar to the argument above, it follows from Lemma \ref{lem2.8} that the first order Melnikov function of the equation (also system \eqref{eq11}) is
\begin{align*}
  &M_1(\rho)
  =
  \int_{-\pi}^{\pi}
\frac{P_{m}(\rho\cos\theta,\rho\sin\theta)\cdot\cos\theta
+Q_{m}(\rho\cos\theta,\rho\sin\theta)\cdot\sin\theta}
  {\prod_{i=1}^{n_1}(\rho\cos\theta-a_i)\prod_{j=1}^{n_2}(\rho\sin\theta-b_j)}d\theta,\\
  &\qquad\qquad \rho\in\big(0,\min_{i,j}\{|a_i|,|b_j|\}\big).
\end{align*}

In order to illustrate our result more clearly,
here we use directly the explicit expression of $M_1$ given in \cite{taubes20}, that is
\begin{align*}
M_1(\rho)
=&
\rho^{-1}\prod_{d\in D}(d-\rho^2)^{-1}
\cdot\left(
\sum_{a\in A} P_{a,\left[\frac{m}{2}\right]+l}(\rho^2)\big(a^2-\rho^2\big)^{-\frac{1}{2}}
+R_{\left[\frac{m-1}{2}\right]+l}(\rho^2)
\right),
\end{align*}
where $P_{a,k}$, $R_k$ represent the polynomials of degree $k$, and
\begin{align*}
  &D=\left\{a^2_i+b^2_j:i=1,\cdots,n_1,j=1,\cdots,n_2\right\},\ \qquad l=\text{Card}(D),\\
  &A=\{|a_i|:i=1,\cdots,n_1\}\cup\{|b_j|:j=1,\cdots,n_2\}.
\end{align*}
Then, taking the transformation $z=a_{\min}^2-\rho^2$ with $a_{\min} =\min\{a:a\in A\}$, $M_1$ can be written as
\begin{align}\label{eq45}
  \rho\prod_{d\in D}(d-\rho^2)M_1(\rho)
  =
  \sum\nolimits_{\sqrt{\hat a+a^2_{\min}}\in A}\left(\hat P_{a,\left[\frac{m}{2}\right]+l}(z)(z+\hat a)^{-\frac{1}{2}}
+\hat R_{\left[\frac{m-1}{2}\right]+l}(z)\right),
\end{align}
where $\hat P_{a,k}$ and $\hat R_k$ are polynomials of degree $k$.

We recall that in \cite{taubes20} the zeros of $M_1$ was studied using the expression \eqref{eq45} and the Derivation-Division algorithm.
By contrast note that the right hand side of equality \eqref{eq45} is a linear combination of the elements of the family
\begin{align*}
\left\{1,z,\cdots,z^{\left[\frac{m-1}{2}\right]+l}\right\}
\bigcup
\left(\bigcup_{\hat a+a_{\min}\in A}\left\{(z+\hat a)^{-\frac{1}{2}},z(z+\hat a)^{-\frac{1}{2}},\cdots,z^{\left[\frac{m}{2}\right]+l}(z+\hat a)^{-\frac{1}{2}}\right\}\right).
\end{align*}
Hence, according to statement (ii) of Proposition \ref{prop1}, the number of isolated zeros of $M_1$ on $\big(0,\min_{i,j}\{|a_i|,|b_j|\}\big)$ (counted with multiplicities), is at most
\begin{align}\label{eq46}
  \text{Card}(A)\cdot\left(\left[\frac{m}{2}\right]+l+1\right)+\left[\frac{m-1}{2}\right]+l.
\end{align}
Observe that $\text{Card}(A)\leq n_1+n_2$ and $l\leq n_1n_2$. Thus this bound does not exceed
\begin{align}\label{eq47}
  (n_1+n_2)\cdot\left(\left[\frac{m}{2}\right]+n_1n_2+1\right)+\left[\frac{m-1}{2}\right]+n_1n_2.
\end{align}
\begin{remark}\label{rem3}
  {\rm
  The numbers in \eqref{eq46} and \eqref{eq47} are exactly the upper bounds given in \cite{taubes20}, which are the main result of that work.
}
\end{remark}

\subsection{Piecewise smooth planar differential systems}\label{an
appropriate label}
It is easy to see that the unperturbed system \eqref{eq10}$|_{\varepsilon=0}$ has a center at the origin with the first integral $x^2+y^2$ and the period annulus
$
  \left\{(x,y)\in\mathbb R^2:0<x^2+y^2<|a|^{-1}\right\}.
$
In order to obtain the first order Melnikov function
we again use the polar coordinates and write the perturbed system in
\begin{align}\label{eq30}
  &\frac{dr}{d\theta}
  =
  \varepsilon
  \frac{P_{m}(r\cos\theta,r\sin\theta)\cdot\cos\theta+Q_{m}(r\cos\theta,r\sin\theta)\cdot\sin\theta}
  {1-ar\cos\theta}
  +o(\varepsilon),
\end{align}
where
\begin{align}\label{eq25}
\begin{split}
  &\left(
\begin{aligned}
&P_m(r\cos\theta,r\sin\theta)\\
&Q_m(r\cos\theta,r\sin\theta)
\end{aligned}
\right)\\
  &\indent=
  \left\{
                  \begin{aligned}
                   &\left(
\begin{aligned}
&P_{m,0}(r\cos\theta,r\sin\theta)\\
&Q_{m,0}(r\cos\theta,r\sin\theta)
\end{aligned}
\right),
                   & \theta\in E_0:=(\vartheta_0,\vartheta_1),\ r\in\big(0,|a|^{-1}\big),\qquad \qquad  \\
                     &\qquad\qquad \vdots & \vdots \qquad \qquad \qquad\qquad \qquad  \\
                   &\left(
\begin{aligned}
&P_{m,n-1}(r\cos\theta,r\sin\theta)\\
&Q_{m,n-1}(r\cos\theta,r\sin\theta)
\end{aligned}
\right),
                   & \theta\in E_{n-1}:=(\vartheta_{n-1},\vartheta_n),\ r\in\big(0,|a|^{-1}\big),\qquad  \\
                   &\left(
\begin{aligned}
&P_{m,n}(r\cos\theta,r\sin\theta)\\
&Q_{m,n}(r\cos\theta,r\sin\theta)
\end{aligned}
\right),
                   & \theta\in E_n:=(\vartheta_n,\pi)\cup[-\pi,\vartheta_0),\ r\in\big(0,|a|^{-1}\big).
                  \end{aligned}
                \right.
\end{split}
\end{align}

For compactness of this paper, we arrange some related results, such as Proposition \ref{prop6.1}, Lemma \ref{lem6.1} and their proofs in Appendix A.2. Applying Lemma \ref{lem2.8} and Proposition \ref{prop6.1} yields the first order Melnikov function of equation \eqref{eq30}
\begin{align}\label{eq42}
\begin{split}
  M_1(\rho)
  =
  &\sum_{s=0}^{n}\int_{E_s}
\frac{P_{m,s}(\rho\cos\theta,\rho\sin\theta)\cdot\cos\theta
+Q_{m,s}(\rho\cos\theta,\rho\sin\theta)\cdot\sin\theta}
  {1-a\rho \cos\theta}d\theta\\
  =
  &
  \sum_{s=0}^{n}\left(
  \sum_{(i,p)\in B_1}c^s_{i,i+2p-1}\rho^{i+2p-1} \int_{E_s}
  \frac{\cos^{i}\theta}{1-a\rho \cos\theta}dt\right.\\
  &\qquad \indent\indent\indent
  +\left.\sum_{(i,p)\in B_2}d^s_{i,i+2p}\rho^{i+2p} \int_{E_s}
  \frac{\sin\theta\cos^{i}\theta}{1-a\rho \cos\theta}dt
  \right),
\end{split}
\end{align}
where $\rho\in(0,|a|^{-1})$, the coefficients $c^s_{i,i+2p-1}$'s and $d^s_{i,i+2p}$'s are determined by $(P_{m,s},Q_{m,s})$, and
\begin{align}\label{eq41}
\begin{split}
  B_1
  =&
  \left\{(0,p):p=1,\cdots,\left[\frac{m+1}{2}\right]\right\}\\
  &\qquad \quad \bigcup\left\{(i,p):i=1,\cdots,m+1, \ p=0,\cdots,\left[\frac{m-i+1}{2}\right]\right\},\\
  B_2
  =&
  \left\{(i,p):i=0,\cdots,m, \ p=0,\cdots,\left[\frac{m-i}{2}\right]\right\}.
\end{split}
\end{align}
We stress that all these coefficients $c^s_{i,i+2p-1}$'s and $d^s_{i,i+2p}$'s can be chosen arbitrarily due to Proposition \ref{prop6.1} and Lemma \ref{lem6.1} in the appendix.

In order to simplify $M_1$ we introduce two more notations
\begin{align}\label{eq4.1}
 \begin{array}{l}
 \displaystyle
 C^{E}_{k,\alpha}(y):=\int_{E}\frac{\cos^k\theta}{(y-\cos\theta)^{\alpha}}d\theta,\\
 \displaystyle
 S^{E}_{k,\alpha}(y):=\int_{E}\frac{\sin\theta\cos^k\theta}{(y- \cos\theta)^{\alpha}}d\theta,
 \end{array}\qquad
   E\subset\mathbb R,\ k\in\mathbb Z^{+}_{0},\ \alpha\in\mathbb R.
\end{align}
Also for convenience we will treat the coefficients as parameters, setting
\begin{align*}
\bm{\mu}
&=
\big(c^s_{i,i+2p-1},d^s_{j,j+2q};s=0,\cdots,n,(i,p)\in B_1,(j,q)\in B_2\big)\\
&\quad \in
\mathbb R^{(n+1)\left({\rm Card(B_1)}+{\rm Card(B_2)}\right)}.
\end{align*}
Then, taking $y=(a\rho)^{-1}$, we get from a direct calculation that $M_1(\rho)=y^{1-m}\hat M_1(y;\bm{\mu})$, where
\begin{align*}
\hat M_1&(y;\bm{\mu})\\
&=
\sum^{n}_{s=0}
\left(
\sum_{(i,p)\in B_1}
c^s_{i,i+2p-1}\cdot y^{m-(i+2p-1)} C^{E_s}_{i,1}(y)
+\sum_{(i,p)\in B_2}
  d^s_{i,i+2p}\cdot y^{m-(i+2p)} S^{E_s}_{i,1}(y)
\right).
\end{align*}
Observe that the zeros of $M_1$ on $(0,|a|^{-1})$, one to one, corresponds to the zeros of $\hat M_1$ on $(1,+\infty)$ (resp. $(-\infty,-1)$) when $a>0$ (resp. $a<0$).

Now let us define two spans
  \begin{align}\label{eq44}
  \begin{split}
    \mathcal B_1
    :=
    &
    \sum_{s=0}^{n}\left(
    \left\langle
    y^{m-(i+2p-1)} C^{E_s}_{i,1}; (i,p)\in B_1
    \right\rangle
    +
    \left\langle
    y^{m-(i+2p)} S^{E_s}_{i,1}; (i,p)\in B_2
    \right\rangle
    \right),\\
    \mathcal B_2
    :=
    &\left\langle1,y,\cdots,y^{m-1}\right\rangle\\
    &+
    \sum_{s=0}^{n}\left(
    \left\langle
    C^{E_s}_{m+1-2p,1}; p=0,\cdots,\left[\frac{m+1}{2}\right]
    \right\rangle
    +
    \left\langle
    S^{E_s}_{m-2p,1}; p=0,\cdots,\left[\frac{m}{2}\right]
    \right\rangle
    \right).
    \end{split}
  \end{align}
Then one can immediately get that $\hat M_1\in\mathcal B_1$. In addition, following Proposition \ref{prop6} stated in Appendix A.2, we actually have
$
\hat M_1\in\mathcal B_1=\mathcal B_2.
$
Hence, there exists a surjective map
$$\mathcal T:
\mathbb R^{(n+1)\left({\rm Card(B_1)}+{\rm Card(B_2)}\right)}
\rightarrow
\mathbb R^{m+(n+1)\left(\left[\frac{m+1}{2}\right]+\left[\frac{m+1}{2}\right]+2\right)},$$
such that
$
\left(\bm{\zeta}_{m-1},\bm{\eta}_{s,\left[\frac{m+1}{2}\right]},\bm{\lambda}_{s,\left[\frac{m}{2}\right]}; s=0,\cdots,n\right)
=
\mathcal T(\bm{\mu})
$
satisfies
\begin{align}\label{eq38}
\hat M_1(y;\bm{\mu})
=
\sum_{i=0}^{m-1}\zeta_i y^i
+\sum^{n}_{s=0}
\left(
\sum_{p=0}^{\left[\frac{m+1}{2}\right]}
\eta_{s,p} C^{E_s}_{m+1-2p,1}(y)
+\sum_{p=0}^{\left[\frac{m}{2}\right]}
\lambda_{s,p} S^{E_s}_{m-2p,1}(y)
\right).
\end{align}

Note that for $l\in\mathbb Z^+_0$,
  \begin{align}\label{eq40}
 \begin{split}
  &\left(C^{E_s}_{k,\alpha}\right)^{(l)}
  =
  (-1)^l\prod_{i=0}^{l-1}(\alpha+i) C^{E_s}_{k,\alpha+l},\ \
  \left(S^{E_s}_{k,\alpha}\right)^{(l)}
  =
  (-1)^l\prod_{i=0}^{l-1}(\alpha+i) S^{E_s}_{k,\alpha+l}.
\end{split}
\end{align}
Therefore we get by \eqref{eq38} that
\begin{align}\label{eq36}
\hat M_1^{(m)}(y;\bm{\mu})
=
(-1)^m m!
\sum^{n}_{s=0}
\left(
\sum_{p=0}^{\left[\frac{m+1}{2}\right]}
\eta_{s,p} C^{E_s}_{m+1-2p,m+1}(y)
+\sum_{p=0}^{\left[\frac{m}{2}\right]}
\lambda_{s,p} S^{E_s}_{m-2p,m+1}(y)
\right).
\end{align}

On the other hand, from \eqref{eq24} and \eqref{eq37} in Appendix A.1, we obtain  for each $E_s$ the following equalities of the spans:
\begin{align}\label{eq39}
\begin{split}
    &\left\langle
    C^{E_s}_{m+1-2p,m+1}(y);\ p=0,\cdots,\left[\frac{m+1}{2}\right]
    \right\rangle\\
    &\indent=
    \left\langle
    y^{-(m+1)}\mathcal C^{E_s}_{m+1-2p}(y^{-1});\ \xi_{E_s}=1, \alpha=m+1, \nu=1, p=0,\cdots,\left[\frac{m+1}{2}\right]
    \right\rangle,\\
    &\left\langle
    S^{E_s}_{m-2p,m+1}(y);p=0,\cdots,\left[\frac{m}{2}\right]
    \right\rangle\\
    &\indent=
    \left\langle
    y^{-(m+1)}\mathcal S^{E_s}_{m+1-2p}(y^{-1});\ \xi_{E_s}=1, \alpha=m+1, \nu=1, p=0,\cdots,\left[\frac{m}{2}\right]
    \right\rangle.
\end{split}
  \end{align}
Consequently, there exists a bijective map
$$
\widehat{\mathcal T}
:
\mathbb R^{(n+1)\left(\left[\frac{m+1}{2}\right]+\left[\frac{m}{2}\right]+2\right)}
\rightarrow
\mathbb R^{(n+1)\left(\left[\frac{m+1}{2}\right]+\left[\frac{m}{2}\right]+2\right)},
$$
such that
\[
\left(\hat{\bm{\eta}}_{s,\left[\frac{m+1}{2}\right]},\hat{\bm{\lambda}}_{s,\left[\frac{m}{2}\right]}; s=0,\cdots,n\right)
=
\widehat{\mathcal T}
\left(\bm{\eta}_{s,\left[\frac{m+1}{2}\right]},\bm{\lambda}_{s,\left[\frac{m}{2}\right]}; s=0,\cdots,n\right)
=
\widehat{\mathcal T}\circ\mathcal T(\bm{\mu}).
\]
That is, the equality
\begin{align}\label{eq32}
\begin{split}
y^{m+1}\hat M_1^{(m)}(y;\bm{\mu})
=
\sum^{n}_{s=0}
\left(
\sum_{p=0}^{\left[\frac{m+1}{2}\right]}
\hat{\eta}_{s,p} \mathcal C^{E_s}_{m+1-2p}(y^{-1})
+\sum_{p=0}^{\left[\frac{m}{2}\right]}
\hat{\lambda}_{s,p} \mathcal S^{E_s}_{m+1-2p}(y^{-1})
\right)
\end{split}
\end{align}
holds, where
 $\mathcal C^{E_s}_{m+1-2p}$ and $\mathcal S^{E_s}_{m+1-2p}$ are those defined in \eqref{eq4} by taking $\xi_{E_s}=1$, $\alpha=m+1$ and $\nu=1$.

We can estimate the number of isolated zeros of $M_1$ (i.e., $\hat M_1$) now. We limit to some number of separation redials, where four of them have appeared in the references \cite{taubes13,taubes14,taubes15,taubes15.3,taubes15.13,taubes15.22,taubes29} (see statements (i)-(iv) of Proposition \ref{prop8} and Figure \ref{fig1}), and two new ones (statements (v) and (vi)) are added now for the first time, which exhibits some clues that how the maximum number of isolated zeros of $M_1$ can be affected by the number $n$ and the distribution of the angles $\vartheta_0,\cdots,\vartheta_n$.

\begin{figure}[h]
\begin{center}
  \includegraphics[width=12.8cm]{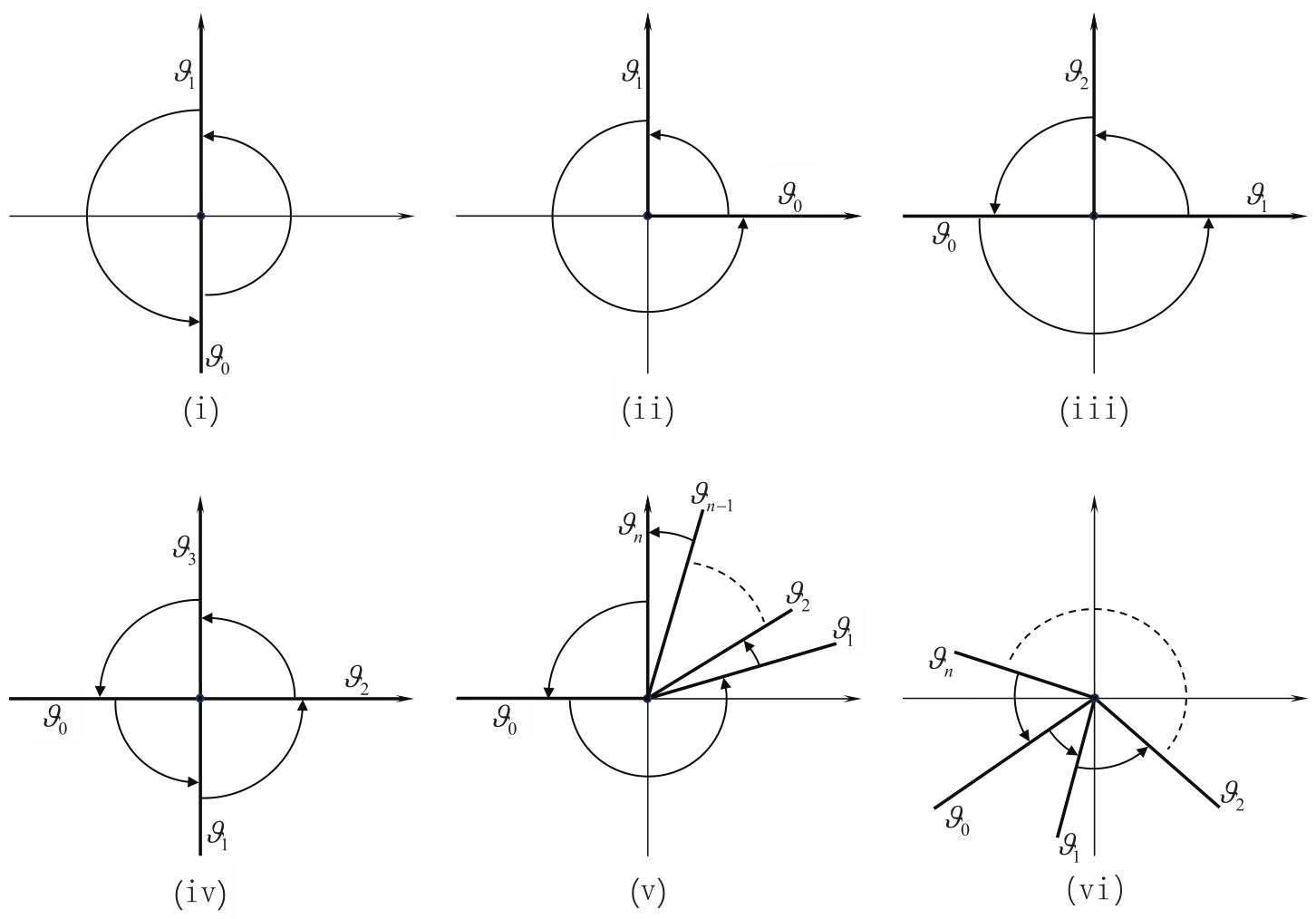}\\
  \caption{Separation radials studied in statements (i)-(vi) of Proposition \ref{prop8}.}\label{fig1}
\end{center}
\end{figure}


\begin{proposition}\label{prop8}
  Let $Z(M_1)$ be the maximum number of isolated zeros of $M_1$ on $(0,|a|^{-1})$, counted with multiplicities. The following statements hold.
  \begin{itemize}
    \item [(i)] $Z(M_1)=2\left[\frac{m+1}{2}\right]+m+1$ if $n=1$ and $\vartheta_0=-\frac{\pi}{2}$, $\vartheta_1=\frac{\pi}{2}$ $($i.e., $E_0=(-\frac{\pi}{2},\frac{\pi}{2})$, $E_1=(\frac{\pi}{2},\pi)\cup[-\pi,-\frac{\pi}{2})$$)$.
    \item [(ii)] $Z(M_1)=2\left[\frac{m+1}{2}\right]+\left[\frac{m}{2}\right]+m+2$ if $n=1$ and $\vartheta_0=0$, $\vartheta_1=\frac{\pi}{2}$ $($i.e., $E_0=(0,\frac{\pi}{2})$, $E_1=(\frac{\pi}{2},\pi)\cup[-\pi,0)$$)$.
    \item [(iii)] $Z(M_1)=2\left(\left[\frac{m+1}{2}\right]+\left[\frac{m}{2}\right]\right)+m+3$ if $n=2$ and $\vartheta_0=-\pi$, $\vartheta_1=0$, $\vartheta_2=\frac{\pi}{2}$ $($i.e., $E_0=(-\pi,0)$, $E_1=(0,\frac{\pi}{2})$, $E_2=(\frac{\pi}{2},\pi)$$)$.
    \item [(iv)] $Z(M_1)=2\left(\left[\frac{m+1}{2}\right]+\left[\frac{m}{2}\right]\right)+m+3$ if $n=3$ and $\vartheta_0=-\pi$, $\vartheta_1=-\frac{\pi}{2}$, $\vartheta_2=0$, $\vartheta_3=\frac{\pi}{2}$ $($i.e., $E_0=(-\pi,-\frac{\pi}{2})$, $E_1=(-\frac{\pi}{2},0)$, $E_2=(0,\frac{\pi}{2})$, $E_3=(\frac{\pi}{2},\pi)$$)$.
    \item [(v)]  $Z(M_1)=n\big(\left[\frac{m+1}{2}\right]+\left[\frac{m}{2}\right]+2\big)+\left[\frac{m+1}{2}\right]+m$ if $n\geq2$ and $\vartheta_0=-\pi$, $\vartheta_1\in(0,\frac{\pi}{2})$ and $\vartheta_{n}=\frac{\pi}{2}$ $($i.e., $E_0=(-\pi,\vartheta_1)$, $E_1=(\vartheta_1,\vartheta_2)$, $\cdots$, $E_{n-1}=(\vartheta_{n-1},\frac{\pi}{2})$, $E_{n}=(\frac{\pi}{2},\pi)$$)$.
    \item [(vi)] $Z(M_1)\leq({\rm Card}(\varTheta)-1)\big(\left[\frac{m+1}{2}\right]+\left[\frac{m}{2}\right]+2\big)+m-1$ for general $n\in \mathbb Z^+$ and $\vartheta_0<\vartheta_1<\cdots<\vartheta_n\in[-\pi,\pi)$, where $\varTheta=\{0,\frac{\pi}{2},\pi,|\vartheta_0|,\cdots,|\vartheta_n|\}$.
  \end{itemize}
\end{proposition}

\begin{proof}
Denote by $Z(\hat M_1)$ the maximum number of isolated zeros of $\hat M_1$ on $(1,+\infty)$ (resp. $(-\infty,-1)$) when $a>0$ (resp. $a<0$). Then the arguments presented before this proposition shows that $Z(M_1)=Z(\hat M_1)$. In what follows we will focus on $\hat M_1$ and prove the statements one by one. Also without loss of generality we suppose that $a>0$.

(i) First by the parities of cosine and sine functions, it is easy to check that
\begin{align*}
\mathcal C^{E_0}_{k}=2\mathcal C^{E_0^+}_{k},\ \ \mathcal C^{E_1}_{k}=2\mathcal C^{E_1^+}_{k},\ \
\mathcal S^{E_0}_{k}=\mathcal S^{E_1}_{k}=0,
\end{align*}
where $E^+_0=(0,\frac{\pi}{2})$ and $E^+_1=(\frac{\pi}{2},\pi)$. Then the equality \eqref{eq32} becomes
\begin{align*}
y^{m+1}&\hat M_1^{(m)}(y;\bm{\mu})
=
2\sum^{1}_{s=0}
\sum_{p=0}^{\left[\frac{m+1}{2}\right]}
\hat{\eta}_{s,p} \mathcal C^{E^+_s}_{m+1-2p}(y^{-1}).
\end{align*}
Observe that $y\mapsto y^{-1}$ is a diffeomorphism from $(1,+\infty)$ to  $(0,1)$. Therefore, it follows from statement (i) of Proposition \ref{prop1} (the second ECT-system in \eqref{eq6}) that
$\hat M_1^{(m)}$ has at most $2\left[\frac{m+1}{2}\right]+1$ zeros on $(1,+\infty)$, counted with multiplicities. This means that $Z(\hat M_1)\leq 2\left[\frac{m+1}{2}\right]+m+1$.

On the other hand, statement (i) of Proposition \ref{prop1} ensures that the functions in the set $\left\{\mathcal C^{E^+_s}_{m+1-2p}(y^{-1}): s=0,1, p=0,\cdots,\left[\frac{m+1}{2}\right]\right\}$ are linearly independent on $(1,+\infty)$, and so does the functions in the set  $\left\{C^{E_s}_{m+1-2p,m+1}(y): s=0,1, p=0,\cdots,\left[\frac{m+1}{2}\right]\right\}$ taking \eqref{eq39} into account. Hence, we have by \eqref{eq40} that the functions in the set $$\{1,y,\cdots,y^{m-1}\}\cup\left\{C^{E_s}_{m+1-2p,1}(y): s=0,1, p=0,\cdots,\left[\frac{m+1}{2}\right]\right\}$$
are also linearly independent on $(1,+\infty)$. From \eqref{eq38}, by applying Theorem \ref{thm2.5}, Lemma \ref{lem2.2} and the surjectivity of $\mathcal T$, there exists $\bm{\mu}$ such that $\hat M_1(y;\bm{\mu})$ has exactly $2\left[\frac{m+1}{2}\right]+m+1$ simple zeros on $(1,+\infty)$. Consequently, $Z(M_1)=Z(\hat M_1)= 2\left[\frac{m+1}{2}\right]+m+1$.

(ii) Again using the parities of cosine and sine functions we have
\begin{align*}
\mathcal C^{E_1}_{k}=2\mathcal C^{E_1^+}_{k}+\mathcal C^{E_0}_{k},\ \
\mathcal S^{E_1}_{k}=-\mathcal S^{E_0}_{k},
\end{align*}
where $E^+_1=(\frac{\pi}{2},\pi)$. Then the equality \eqref{eq32} is reduced to
\begin{align*}
\begin{split}
y^{m+1}\hat M_1^{(m)}(y;\bm{\mu})
&=
\sum_{p=0}^{\left[\frac{m+1}{2}\right]}
\left(
(\hat{\eta}_{0,p}+\hat{\eta}_{1,p}) \mathcal C^{E_0}_{m+1-2p}(y^{-1})
+2\hat{\eta}_{1,p} \mathcal C^{E_1^+}_{m+1-2p}(y^{-1})
\right)\\
&\indent+\sum_{p=0}^{\left[\frac{m}{2}\right]}
(\hat{\lambda}_{0,p}-\hat{\lambda}_{1,p}) \mathcal S^{E_0}_{m+1-2p}(y^{-1}).
\end{split}
\end{align*}
Similar to the arguments in statement (i), we obtain by statement (i) of Proposition \ref{prop1} that $\hat M_1^{(m)}(y;\bm{\mu})$ has at most $2\left[\frac{m+1}{2}\right]+\left[\frac{m}{2}\right]+2$ zeros on $(1,+\infty)$, and therefore $Z(\hat M_1)\leq 2\left[\frac{m+1}{2}\right]+\left[\frac{m}{2}\right]+m+2$. In addition, we also get that the functions in the set
\begin{align*}
  &\left\{C^{E_0}_{m+1-2p,m+1}(y): p=0,\cdots,\left[\frac{m+1}{2}\right]\right\}\\
&\qquad   \cup\left\{C^{E_1^+}_{m+1-2p,m+1}(y): p=0,\cdots,\left[\frac{m+1}{2}\right]\right\}
\cup\left\{S^{E_0}_{m-2p,m+1}(y): p=0,\cdots,\left[\frac{m}{2}\right]\right\}
\end{align*}
are linearly independent on $(1,+\infty)$, and so does the functions in the set
\begin{align*}
  &\{1,y,\cdots,y^{m-1}\}\cup\left\{C^{E_0}_{m+1-2p,1}(y): p=0,\cdots,\left[\frac{m+1}{2}\right]\right\}\\
  &\qquad\cup\left\{C^{E_1}_{m+1-2p,1}(y): p=0,\cdots,\left[\frac{m+1}{2}\right]\right\}
\cup\left\{S^{E_0}_{m-2p,1}(y): p=0,\cdots,\left[\frac{m}{2}\right]\right\}.
\end{align*}
Then, thanks to \eqref{eq38} and the surjectivity of $\mathcal T$, we can assert that there exists $\bm{\mu}$ such that $\hat M_1(y;\bm{\mu})$ has exactly $2\left[\frac{m+1}{2}\right]+\left[\frac{m}{2}\right]+m+2$ simple zeros on $(1,+\infty)$. Accordingly, the assertion holds.

(iii) The conclusion follows exactly from the same arguments as before, with the observation that
\begin{align*}
&\mathcal C^{E_0}_{k}=\mathcal C^{E_1}_{k}+\mathcal C^{E_2}_{k},\ \
\mathcal S^{E_0}_{k}=-\left(\mathcal S^{E_1}_{k}+\mathcal S^{E_2}_{k}\right)
\end{align*}
 and
\begin{align*}
\begin{split}
y^{m+1}\hat M_1^{(m)}(y;\bm{\mu})
&=
\sum_{p=0}^{\left[\frac{m+1}{2}\right]}
\left(
(\hat{\eta}_{0,p}+\hat{\eta}_{1,p}) \mathcal C^{E_1}_{m+1-2p}(y^{-1})
+(\hat{\eta}_{0,p}+\hat{\eta}_{2,p}) \mathcal C^{E_2}_{m+1-2p}(y^{-1})
\right)\\
&\indent+\sum_{p=0}^{\left[\frac{m}{2}\right]}
\left(
(\hat{\lambda}_{1,p}-\hat{\lambda}_{0,p}) \mathcal S^{E_1}_{m+1-2p}(y^{-1})
+(\hat{\lambda}_{2,p}-\hat{\lambda}_{0,p}) \mathcal S^{E_2}_{m+1-2p}(y^{-1})
\right).
\end{split}
\end{align*}

(iv) In this case we have that
\begin{align*}
&\mathcal C^{E_0}_{k}=\mathcal C^{E_3}_{k},\ \ \mathcal C^{E_1}_{k}=\mathcal C^{E_2}_{k},\ \
\mathcal S^{E_0}_{k}=-\mathcal S^{E_3}_{k},\ \ \mathcal S^{E_1}_{k}=-\mathcal S^{E_2}_{k},
\end{align*}
which implies
\begin{align*}
\begin{split}
y^{m+1}\hat M_1^{(m)}(y;\bm{\mu})
&=
\sum_{p=0}^{\left[\frac{m+1}{2}\right]}
\left(
(\hat{\eta}_{1,p}+\hat{\eta}_{2,p}) \mathcal C^{E_2}_{m+1-2p}(y^{-1})
+(\hat{\eta}_{0,p}+\hat{\eta}_{3,p}) \mathcal C^{E_3}_{m+1-2p}(y^{-1})
\right)\\
&\indent+\sum_{p=0}^{\left[\frac{m}{2}\right]}
\left(
(\hat{\lambda}_{2,p}-\hat{\lambda}_{1,p}) \mathcal S^{E_2}_{m+1-2p}(y^{-1})
+(\hat{\lambda}_{3,p}-\hat{\lambda}_{0,p}) \mathcal S^{E_3}_{m+1-2p}(y^{-1})
\right).
\end{split}
\end{align*}
Then exactly as we did in the previous cases, one can obtain that $Z(M_1)=Z(\hat M_1)=2\left(\left[\frac{m+1}{2}\right]+\left[\frac{m}{2}\right]\right)+m+3$.

(v) In the first step of the process one can check that
\begin{align*}
\mathcal C^{E_0}_{k}=2\mathcal C^{E_0^+}_{k}+\sum_{s=1}^{n}\mathcal C^{E_s}_{k},\ \
\mathcal S^{E_0}_{k}=-\sum_{s=1}^{n}\mathcal S^{E_s}_{k},
\end{align*}
where $E_0^+=(0,\vartheta_1)$. Then,
\begin{align*}
\begin{split}
y^{m+1}\hat M_1^{(m)}(y;\bm{\mu})
&=
\sum_{p=0}^{\left[\frac{m+1}{2}\right]}
\left(
2\hat{\eta}_{0,p} \mathcal C^{E_0^+}_{m+1-2p}(y^{-1}) + \sum_{s=1}^{n}(\hat\eta_{0,p}+\hat{\eta}_{s,p}) \mathcal C^{E_s}_{m+1-2p}(y^{-1})
\right)\\
&\indent+\sum_{p=0}^{\left[\frac{m}{2}\right]}
\left(
\sum_{s=1}^{n}(\hat{\lambda}_{s,p}-\hat{\lambda}_{0,p}) \mathcal S^{E_s}_{m+1-2p}(y^{-1})
\right).
\end{split}
\end{align*}
Observe that $E^+_0$ and $E_1,\cdots,E_n$ are non-intersecting and are contained in $(0,\frac{\pi}{2})\cup(\frac{\pi}{2},\pi)$. Hence using the argument as above,
$Z(M_1)=Z(\hat M_1)=n\big(\left[\frac{m+1}{2}\right]+\left[\frac{m}{2}\right]+2\big)+\left[\frac{m+1}{2}\right]+m$.

(vi) In the general case, we have by the parities of cosine and sine functions that
\begin{align}\label{eq28}
 \begin{split}
    \mathcal C^{E_s}_k&=\mathcal C^{E_s\cap[-\pi,\pi)}_k=\mathcal C^{E_{s,1}}_k+\mathcal C^{E_{s,2}}_k+\mathcal C^{E_{s,3}}_k+\mathcal C^{E_{s,4}}_k,\\
    \mathcal S^{E_s}_k&=\mathcal S^{E_s\cap[-\pi,\pi)}_k=\mathcal S^{E_{s,1}}_k+\mathcal S^{E_{s,2}}_k-\mathcal S^{E_{s,3}}_k-\mathcal S^{E_{s,4}}_k,
 \end{split}
  \end{align}
  where $s=0,\cdots,n$ and
 \[
  \begin{array}{ll}
  E_{s,1}=E_{s}\cap(0,\frac{\pi}{2}), & \ \ \  E_{s,2}=E_s\cap(\frac{\pi}{2},\pi),\\
    E_{s,3}=\left\{\theta:-\theta\in E_s\cap(-\pi,-\frac{\pi}{2})\right\}, &\ \ \ E_{s,4}=\left\{\theta:-\theta\in E_s\cap(-\frac{\pi}{2},0)\right\}.
  \end{array}
  \]

  Clearly, the end points of $E_{s,1}$, $E_{s,2}$, $E_{s,3}$ and $E_{s,4}$ are all contained in the set $\varTheta$. We reorder all the elements in $\varTheta$, writing as $0=\tilde\vartheta_0<\tilde\vartheta_1<\cdots<\tilde\vartheta_{\tilde n}=\pi$ with $\tilde n:={\rm Card}(\varTheta)-1$, and define
  $\tilde E_0=(\tilde\vartheta_0,\tilde\vartheta_{1}),\cdots,\tilde E_{\tilde n-1}=(\tilde\vartheta_{\tilde n-1},\tilde\vartheta_{\tilde n})$. Then each of $E_{s,1}\backslash \varTheta$, $E_{s,2}\backslash \varTheta$, $E_{s,3}\backslash \varTheta$ and $E_{s,4}\backslash \varTheta$ is a union of some intervals $\tilde E_i$, $i\in\{0,1,\ldots,\tilde n-1\}$. Accordingly, taking \eqref{eq28} into account, each $\mathcal C^{E_s}_k$ (resp. $\mathcal S^{E_s}_k$) is a linear combination of $\mathcal C^{\tilde E_0}_k,\cdots,\mathcal C^{\tilde E_{\tilde n}}_k$ (resp. $\mathcal S^{\tilde E_0}_k,\cdots,\mathcal S^{\tilde E_{\tilde n}}_k$). This means that  equality \eqref{eq32} can be rewritten as
  \begin{align*}
y^{m+1}\hat M_1^{(m)}(y;\bm{\mu})
&=
\sum^{\tilde n-1}_{s=0}
\left(
\sum_{p=0}^{\left[\frac{m+1}{2}\right]}
\tilde{\eta}_{s,p} \mathcal C^{\tilde E_s}_{m+1-2p}(y^{-1})
+\sum_{p=0}^{\left[\frac{m}{2}\right]}
\tilde{\lambda}_{s,p} \mathcal S^{\tilde E_s}_{m+1-2p}(y^{-1})
\right).
\end{align*}
Since $\tilde E_0,\cdots,\tilde E_{\tilde n}\subseteq(0,\frac{\pi}{2})\cup(\frac{\pi}{2},\pi)$,
from statement (i) of Proposition \ref{prop1} one gets that $\hat M_1^{(m)}$ has at most $\tilde n\big(\left[\frac{m+1}{2}\right]+\left[\frac{m}{2}\right]+2\big)-1= ({\rm Card}(\varTheta)-1)\big(\left[\frac{m+1}{2}\right]+\left[\frac{m}{2}\right]+2\big)-1$ zeros in $(1,+\infty)$, counted with multiplicities. Our assertion follows.

It completes the proof of the proposition.
\end{proof}

\begin{remark}
{\rm For Proposition \ref{prop8} we have the next comments.
\begin{itemize}
\item According to the definition of $\varTheta$ in statement (vi) of Proposition \ref{prop8}, it is clear that $Z(M_1)\leq (n+3)\big(\left[\frac{m+1}{2}\right]+\left[\frac{m}{2}\right]+2\big)+m-1$ for the general case. Although the sharpness of this estimate is not obtained, we still can see that the upper bound is essentially the same as the explicit one obtained in statement (v) as $n$ large enough. In any case, the Chevbyshev families used in this proposition give us an efficient and unified way to analyze the Melnikov functions for the system with different kinds of separation redials.

\item On the other hand, statement (vi) of Proposition \ref{prop8} also tells us that the symmetry of $\vartheta_0,\cdots,\vartheta_n$ is a major factor in controlling $Z(M_1)$. Indeed, let $k$ be the cardinality of the set $\{\vartheta_0,\cdots,\vartheta_n\}\cap\{-\vartheta_0,\cdots,-\vartheta_n\}$, i.e., the number of the symmetric elements of $\vartheta_0,\cdots,\vartheta_n$. Then we have $k\leq\left[\frac{n}{2}\right]+1$ and ${\rm Card}(\varTheta)\leq n-k+4$. Accordingly, $Z(M_1)\leq (n-k+3)\big(\left[\frac{m+1}{2}\right]+\left[\frac{m}{2}\right]+2\big)+m-1$.
As an example, in the natural symmetric case where $\vartheta_0,\cdots,\vartheta_n$ are defined by $\exp((n+1)\vartheta\sqrt{-1})=1$, one can check that $k=\left[\frac{n}{2}\right]+1$. Thus, in this case $Z(M_1)$ does not exceed $\big(\left[\frac{n}{2}\right]+3\big)\big(\left[\frac{m+1}{2}\right]+\left[\frac{m}{2}\right]+2\big)+m-1$, which is near one half of the value in the asymmetric case of statement (v). From this perspective, it seems that the symmetry of the separation radials of a piecewise smooth differential system plays an important role in affecting the number of limit cycles  of the given system.
\end{itemize}
}
\end{remark}

\appendix
\section{}\label{an
appropriate label}
\subsection{Chebyshev properties of the order sets in Proposition \ref{prop7}}
According to the Chebyshev polynomials of the first
kind and of the second kind, for any $i\in \mathbb Z^+$, we have the following important equalities between the linear spans
\begin{align}\label{eq24}
\begin{split}
  &\langle 1,\cos\theta,\cos2\theta\cdots,\cos i\theta\rangle
  =
  \langle 1,\cos\theta,\cos^2\theta,\cdots,\cos^i\theta\rangle, \\
  &\langle 1,\cos2\theta,\cos4\theta\cdots,\cos 2i\theta\rangle
  =
  \langle 1,\cos^2\theta,\cos^4\theta,\cdots,\cos^{2i}\theta\rangle, \\
  &\langle \cos\theta,\cos3\theta\cdots,\cos (2i+1)\theta\rangle
  =
  \langle \cos\theta,\cos^3\theta,\cdots,\cos^{2i+1}\theta\rangle \\
  &\indent\indent\indent\indent\indent\indent\indent\indent\indent\indent\indent\ \ \
  =
  \langle \cos\theta,\cos\theta\cos2\theta,\cdots,\cos\theta\cos2i\theta\rangle, \\
  \end{split}
\end{align}
and
\begin{align}\label{eq37}
\begin{split}
  &\langle \sin\theta,\sin2\theta,\cdots,\sin i\theta\rangle
  =
  \langle \sin\theta,\sin\theta\cos\theta,\cdots,\sin\theta\cos^{i-1}\theta\rangle,\\
  &\langle \sin2\theta,\sin4\theta,\cdots,\sin 2i\theta\rangle
  =
  \langle \sin\theta\cos\theta,\sin\theta\cos^3\theta,\cdots,\sin\theta\cos^{2i-1}\theta\rangle\\
  &\indent\indent\indent\indent\indent\indent\indent\indent\indent\indent\indent\ \ \
  =
  \langle \sin\theta\cos\theta,\sin\theta\cos3\theta,\cdots,\sin\theta\cos(2i-1)\theta\rangle,\\
  &\langle \sin\theta,\sin3\theta,\cdots,\sin (2i+1)\theta\rangle
  =
  \langle \sin\theta,\sin\theta\cos^2\theta,\cdots,\sin\theta\cos^{2i}\theta\rangle\\
  &\indent\indent\indent\indent\indent\indent\indent\indent\indent\indent\indent\ \ \
  =
  \langle \sin\theta,\sin\theta\cos2\theta,\cdots,\sin\theta\cos2i\theta\rangle,
\end{split}
\end{align}
respectively. We now utilize them to prove Proposition \ref{prop7}.

\begin{proof}[Proof of Proposition \ref{prop7}]
Since $\cos\theta$ is monotonically decreasing with respect to $\theta$ on $(0,\pi)$, any polynomial in variable
$\cos\theta$ of degree $i$ has at most $i$ isolated zeros (counted with multiplicities) on $(0,\pi)$, and the upper bound is sharp. From definition, the order set
$\{1,\cos\theta,\cdots,\cos^m\theta\}$ is an ECT-system on $(0,\pi)$, and so does $\{1,\cos\theta,\cdots,\cos m\theta\}$ taking \eqref{eq24} into account. Moreover, note that $\sin\theta\neq0$ on $(0,\pi)$. Hence by Lemma \ref{lem2.4}, $\{\sin\theta,\sin\theta\cos\theta,\cdots,\sin\theta\cos^{m-1}\theta\}$ is also an ECT-system on $(0,\pi)$, and so does $\{\sin\theta,\sin2\theta,\cdots,\sin m\theta\}$ applying \eqref{eq37}.

Consider the third order set in \eqref{eq21}. Due to the assertion proved above and Lemma \ref{lem2.2}, its first $m+1$ corresponding Wronskians on $(0,\pi)$ are
\begin{align*}
&W[1]=1\neq0,\\
&W[1,\cos\theta,\cdots,\cos i\theta]\neq0,\indent  i=1,\cdots,m.
\end{align*}
Moreover, we have by a direct calculation that
\begin{align*}
&W[1,\cos\theta,\cdots,\cos m\theta,\sin m\theta,\sin(m-1)\theta,\cdots,\sin i\theta]\\
&\ \ =\left|
           \begin{array}{ccccc}
             -\sin\theta       & -2\sin2\theta   & \cdots     &-m\sin m\theta   &m\cos m\theta \\
             -\cos\theta       & -2^2\cos2\theta & \cdots     &-m^2\cos m\theta &-m^2\sin m\theta\\
             \sin\theta        & 2^3\sin2\theta & \cdots     &m^3\sin m\theta  &-m^3\cos m\theta \\
             \cos\theta        & 2^4\cos2\theta & \cdots     &m^4\cos m\theta  &m^4\sin m\theta \\
             \vdots            & \ddots     &  \vdots       &   \vdots      &    \vdots       \\
             \cos^{(K)}\theta  &2^K\cos^{(K)}2\theta& \cdots     & m^K\cos^{(K)}m\theta & m^K\sin^{(K)}m\theta\\
           \end{array}
         \right.\\
&\qquad \qquad\qquad \qquad \qquad \qquad \ \ \indent \left.
           \begin{array}{ccc}
             (m-1)\cos (m-1)\theta                 & \cdots                             & i\sin i\theta \\
             -(m-1)^2\sin (m-1)\theta              & \cdots                             & -i^2\cos i\theta \\
             -(m-1)^3\cos (m-1)\theta              & \cdots                             & -i^{3}\sin i\theta \\
             (m-1)^4\sin (m-1)\theta               & \cdots                             & i^{4}\cos i\theta \\
                 \vdots                            &   \ddots                           & \vdots \\
            (m-1)^K\sin^{(K)}(m-1)\theta           & \cdots                             &i^{K}\sin^{(K)}i\theta\\
           \end{array}
         \right|\\
&\ \ =\left|
           \begin{array}{ccccc}
             -\sin\theta       & -2\sin2\theta   & \cdots     &-m\sin m\theta   &m\cos m\theta \\
             -\cos\theta       & -2^2\cos2\theta & \cdots     &-m^2\cos m\theta &-m^2\sin m\theta\\
             (1-m^2)\sin\theta        & 2(4-m^2)\sin2\theta & \cdots      &0                &0 \\
             (1-m^2)\cos\theta        & 4(4-m^2)\cos2\theta & \cdots      &0                & \\
             \vdots            & \ddots     &  \vdots       &   \vdots      &    \vdots       \\
             (1-m^2)\cos^{(K)}\theta  &2^{K-2}(4-m^2)\cos^{(K)}2\theta& \cdots  &0                &0\\
           \end{array}
         \right.\\
&\qquad \ \ \indent \left.
           \begin{array}{ccc}
             (m-1)\cos (m-1)\theta                 & \cdots                             & i\sin i\theta \\
             -(m-1)^2\sin (m-1)\theta              & \cdots                             & -i^2\cos i\theta \\
             -(m-1)((m-1)^2-m^2)\cos (m-1)\theta              & \cdots                             & -i(i^2-m^2)\sin i\theta \\
             (m-1)^2((m-1)^2-m^2)\sin (m-1)\theta               & \cdots                             & i^{2}(i^2-m^2)\cos i\theta \\
                 \vdots                            &   \ddots                           & \vdots \\
            (m-1)^{K-2}((m-1)^2-m^2)\sin^{(K)}(m-1)\theta           & \cdots                             &i^{K-2}(i^2-m^2)\sin^{(K)}i\theta\\
           \end{array}
         \right|\\
&\ \ =m^3\cdot\prod\limits_{j=1}^{m-1}(m^{2}-j^2)\cdot\prod_{j=i}^{m-1}(m^{2}-j^2)\\
&\ \ \indent
\cdot\left|
           \begin{array}{cccc}
             -\sin\theta       & -2\sin2\theta   & \cdots     &-(m-1)\sin (m-1)\theta   \\
             -\cos\theta       & -2^2\cos2\theta & \cdots     &-(m-1)^2\cos (m-1)\theta\\
             \sin\theta        & 2^3\sin2\theta & \cdots     &(m-1)^3\sin (m-1)\theta  \\
             \cos\theta        & 2^4\cos2\theta & \cdots     &(m-1)^4\cos (m-1)\theta  \\
             \vdots            & \ddots          &  \vdots       &   \vdots      \\
             \cos^{(K-2)}\theta  &2^{K-2}\cos^{(K-2)}2\theta & \cdots   & (m-1)^{K-2}\cos^{(K-2)}(m-1)\theta\\
           \end{array}
         \right.\\
&\ \ \ \ \indent
\left.
           \begin{array}{cccc}
             (m-1)\cos (m-1)\theta          &(m-2)\cos (m-2)\theta                 & \cdots                             & i\sin i\theta \\
             -(m-1)^2\sin (m-1)\theta       &-(m-2)^2\sin (m-2)\theta              & \cdots                             & -i^2\cos i\theta \\
             -(m-1)^3\cos (m-1)\theta       &-(m-2)^3\cos (m-2)\theta              & \cdots                             & -i^{3}\sin i\theta \\
             (m-1)^4\sin (m-1)\theta        &(m-2)^4\sin (m-2)\theta               & \cdots                             & i^{4}\cos i\theta \\
                 \vdots       &    \vdots                            &   \ddots                           & \vdots \\
              (m-1)^{K-2}\sin^{(K-2)}(m-1)\theta  &(m-2)^{K-2}\sin^{(K-2)}(m-2)\theta           & \cdots                &i^{K-2}\sin^{(K-2)}i\theta\\
           \end{array}
         \right|\\
&\ \ =
m^3\prod\limits_{j=1}^{m-1}(m^{2}-j^2)\cdot\prod_{j=i}^{m-1}(m^{2}-j^2)\\
&\qquad \qquad \ \ \indent
\cdot W[1,\cos\theta,\cdots,\cos (m-1)\theta,\sin (m-1)\theta,\sin(m-2)\theta,\cdots,\sin i\theta],
\end{align*}
where $K=2m-i+1$ and $i=1,\cdots,m.$
These recursions imply that for each fixed $i=1,\cdots,m$ and $\theta\in(0,\pi)$,
\begin{align*}
  &W[1,\cos\theta,\cdots,\cos m\theta,\sin m\theta,\sin(m-1)\theta,\cdots,\sin i\theta]\\
&\indent=
\prod_{l=i}^{m}\left(l^3\prod\limits_{j=1}^{l-1}(l^{2}-j^2)\prod\limits_{j=i}^{l-1}(l^{2}-j^2)\right)
W[1,\cos\theta,\cdots,\cos (i-1)\theta]  \neq 0.
\end{align*}
As a result, Lemma \ref{lem2.2} tells us that the third order set in \eqref{eq21} is an ECT-system on $(0,\pi)$.

Now let us consider the order sets in \eqref{eq22}. By applying the first assertion and taking the change of variable $\theta\mapsto2\theta$ (resp. $\theta\mapsto2\pi-2\theta$) into account, we get that these three order sets are ECT-systems on $(0,\frac{\pi}{2})$ (resp. $(\frac{\pi}{2},\pi)$) when $k=0$.

When $k=1$ we use the third equalities in \eqref{eq24} and \eqref{eq37}.
Note that $\cos\theta\neq0$ and $\sin\theta\neq0$ on $(0,\pi)\backslash\{\frac{\pi}{2}\}$. Hence by definition and Lemma \ref{lem2.4},
the order sets $\{ \cos\theta,\cos3\theta,\cdots,\cos (2m+1)\theta\}$ and $\{ \sin\theta,\sin3\theta\cdots,\sin (2m+1)\theta\}$ are ECT-systems on $(0,\frac{\pi}{2})$ (resp. $(\frac{\pi}{2},\pi)$), if and only if the order set $\{1, \cos2\theta,\cdots,\cos 2m\theta\}$ is an ECT-system on $(0,\frac{\pi}{2})$ (resp. $(\frac{\pi}{2},\pi)$). The argument goes back to the case $k=0$.

Finally, it remains to show that the order set $\{\cos \theta,\cos3\theta,\cdots,\cos(2m+1)\theta,\sin(2m+1)\theta,\sin(2m-1)\theta,\cdots,\sin\theta\}$ is an ECT-system on $(0,\frac{\pi}{2})$ and on $(\frac{\pi}{2},\pi)$. The proof follows similarly as in the previous argument of the first assertion. As a result, we obtain the Chebyshev property for all the order sets stated in \eqref{eq21} and \eqref{eq22}.
\end{proof}

\subsection{Equalities of some spans}
This subsection is devoted to providing a simple observation and two equalities of spans which are useful to simplify the expression of the Melnikov function of system \eqref{eq10}. In fact, although some of the proofs of these results are a little lengthy, they are not technical and follow direct calculations.
\begin{lemma}\label{lem6.1}
  Let $\mathcal V_0$ and $\mathcal{V}_1$ be two linear spans defined by:
  \begin{align*}
    &\mathcal V_0=\langle\varepsilon_0,\varepsilon_1,\cdots,\varepsilon_{m_0}\rangle,\ \
    \mathcal V_1=\langle\xi_0,\xi_1,\cdots,\xi_{m_1}\rangle.
  \end{align*}
  If $\mathcal V_1$ is a subspace of $\mathcal V_0$, then for any given column vector
$\bm b_{m_1}\in\mathbb R^{m_1+1}$, there exists a column vector
 $\bm a_{m_0}\in\mathbb R^{m_0+1}$ such that $\bm{\xi}_{m_1}\bm b_{m_1}=\bm{\varepsilon}_{m_0}\bm a_{m_0}$ $($i.e.,
  $b_0\xi_0+b_1\xi_1+\cdots+b_{m_1}\xi_{m_1}=a_0\varepsilon_0+a_1\varepsilon_1+\cdots+a_{m_0}\varepsilon_{m_0}$$)$.
\end{lemma}
\begin{proof}
  The assertion is trivial because of the fact that $\bm{\xi}_{m_1}\bm b_{m_1}\in\mathcal V_1\subseteq\mathcal V_0$ for any $\bm b_{m_1}\in\mathbb R^{m_1+1}$.
\end{proof}

\begin{proposition}\label{prop6.1}
  For any $m\in\mathbb Z^+_0$ suppose that
\begin{align*}
  &\mathcal V_0
  =
  \left\langle
  \rho^j\frac{\cos^{i+1}\theta\sin^{j-i}\theta}{1-a\rho \cos\theta},
  \rho^j\frac{\cos^{i}\theta\sin^{j+1-i}\theta}{1-a\rho \cos\theta};\
  i=0,\cdots,j,\ j=0,\cdots,m
  \right\rangle,\\
  &\mathcal V_1
  =
  \left\langle
  \rho^{i+2p-1}\frac{\cos^{i}\theta}{1-a\rho \cos\theta},
  \rho^{j+2q}\frac{\sin\theta\cos^{j}\theta}{1-a\rho \cos\theta};\
  (i,p)\in B_1,\ (j,q)\in B_2
  \right\rangle,
\end{align*}
where $B_1$ and $B_2$ are those defined in \eqref{eq41}.
Then $\mathcal V_1=\mathcal V_0$.
\end{proposition}
\begin{proof}
  The proof is only based on some simple and direct calculations. In fact,
  for $(i,p)\in B_1$ we have
  \begin{align*}
    \rho^{i+2p-1}\frac{\cos^{i}\theta}{1-a\rho \cos\theta}
    &=
    \rho^{i+2p-1}\frac{(\sin^2\theta+\cos^2\theta)^p\cos^{i}\theta}{1-a\rho \cos\theta}\\
    &=
    \sum_{k=0}^{p}\binom{p}{k}\rho^{i+2p-1}\frac{\cos^{i+2k}\theta\sin^{2p-2k}\theta}{1-a\rho \cos\theta}\in
    \mathcal V_0.
  \end{align*}
  Also, for $(j,q)\in B_2$ one can verify that
  \begin{align*}
    \rho^{j+2q}\frac{\sin\theta\cos^{j}\theta}{1-a\rho \cos\theta}
    &=
    \rho^{j+2q}\frac{(\sin^2\theta+\cos^2\theta)^q\sin\theta\cos^{j}\theta}{1-a\rho \cos\theta}\\
    &=
    \sum_{k=0}^{q}\binom{q}{k}\rho^{j+2q}\frac{\cos^{j+2k}\theta\sin^{2q-2k+1}\theta}{1-a\rho \cos\theta}\in
    \mathcal V_0.
    \end{align*}
    Accordingly, each generator of $\mathcal V_1$ is an element of $\mathcal V_0$, which means $\mathcal V_1\subseteq\mathcal V_0$.

    On the other hand, suppose that $i,j\in\{0,\cdots,m\}$ with $i\leq j\leq m$. Then when $j-i$ is even, i.e., $j=i+2p$ with $p=0,\cdots,\left[\frac{m-i}{2}\right]$, we obtain
    \begin{align*}
    \rho^j\frac{\cos^{i+1}\theta\sin^{j-i}\theta}{1-a\rho \cos\theta}
    &=
    \rho^{i+2p}\frac{\cos^{i+1}\theta(1-\cos^2\theta)^{p}}{1-a\rho \cos\theta}\\
    &=
    \sum_{k=0}^{p}(-1)^k\binom{p}{k}\rho^{(i+2k+1)+2(p-k)-1}\frac{\cos^{i+2k+1}\theta}{1-a\rho \cos\theta} \in
    \mathcal V_1,\\
    \rho^j\frac{\cos^{i}\theta\sin^{j+1-i}\theta}{1-a\rho \cos\theta}
    &=
    \rho^{i+2p}\frac{\sin\theta\cos^{i}\theta(1-\cos^2\theta)^{p}}{1-a\rho \cos\theta}\\
    &=
    \sum_{k=0}^{p}(-1)^k\binom{p}{k}\rho^{(i+2k)+2(p-k)}\frac{\sin\theta\cos^{i+2k}\theta}{1-a\rho \cos\theta} \in
    \mathcal V_1.
    \end{align*}
    When $j-i$ is odd, i.e., $j=i+2p+1$ with $p=0,\cdots,\left[\frac{m-i-1}{2}\right]$, we get
    \begin{align*}
    \rho^j\frac{\cos^{i+1}\theta\sin^{j-i}\theta}{1-a\rho \cos\theta}
    &=
    \rho^{i+2p+1}\frac{\sin\theta\cos^{i+1}\theta(1-\cos^2\theta)^{p}}{1-a\rho \cos\theta}\\
    &=
    \sum_{k=0}^{p}(-1)^k\binom{p}{k}\rho^{(i+2k+1)+2(p-k)}\frac{\sin\theta\cos^{i+2k+1}\theta}{1-a\rho \cos\theta} \in
    \mathcal V_1,\\
    \rho^j\frac{\cos^{i}\theta\sin^{j+1-i}\theta}{1-a\rho \cos\theta}
    &=
    \rho^{i+2p+1}\frac{\cos^{i}\theta(1-\cos^2\theta)^{p+1}}{1-a\rho \cos\theta}\\
    &=
    \sum_{k=0}^{p+1}(-1)^k\binom{p+1}{k}\rho^{(i+2k)+2(p+1-k)-1}\frac{\cos^{i+2k}\theta}{1-a\rho \cos\theta} \in
    \mathcal V_1.
    \end{align*}
    As a result, $\mathcal V_0\subseteq\mathcal V_1$. The assertion follows.
    \end{proof}

\begin{lemma}\label{lem2.9}
  Suppose that $E\subseteq[-\pi,\pi]$, $k\in\mathbb Z_0^+$ and $l\in\mathbb Z$ with $k+ l\geq0$. The following equality holds.
  \begin{align*}
  y^l\int_{E}
  \frac{\sin^{\kappa}\theta\cos^k\theta}{y-\cos\theta}d\theta
  &=
  \int_{E}
  \frac{\sin^{\kappa}\theta\cos^{k+l}\theta}{y-\cos\theta}d\theta
  +\left(y^lT^E_{\kappa,k-1}(y)-T^E_{\kappa,k+l-1}(y)\right),
  \end{align*}
  where $\kappa\in\{0,1\}$ and $T^E_{\kappa,q}$ is a polynomial having the next expression
  \begin{align}\label{eq29}
  \begin{split}
  T^E_{\kappa,q}(y)
  =
  \left\{
  \begin{aligned}
  &0 , &q=-1,\\
  &\sum_{i=1}^{q+1}\binom{q+1}{i}(-1)^i y^{q+1-i}
  \int_{E}\sin^{\kappa}\theta\left(y-\cos\theta\right)^{i-1}d\theta , &q\geq0.
  \end{aligned}
  \right.
  \end{split}
  \end{align}
\end{lemma}
\begin{proof}
Suppose that $p\in\mathbb Z$ and $q\in\mathbb Z^+_0$. We have
  \begin{align}\label{eq27}
  \begin{split}
  y^p&\int_{E}
  \frac{\sin^{\kappa}\theta\cos^q\theta}{y-\cos\theta}d\theta =
  y^p\int_{E}
  \frac{\sin^{\kappa}\theta(\cos\theta-y+y)^q}{y-\cos\theta}d\theta\\
  &=
  y^p
  \left(
  \sum_{i=1}^{q}\binom{q}{i}(-1)^i y^{q-i}
  \int_{E}\sin^{\kappa}\theta\left(y-\cos\theta\right)^{i-1}d\theta
  +y^q\int_{E}
  \frac{\sin^{\kappa}\theta}{y-\cos\theta}d\theta
  \right)\\
  &=
   y^pT^E_{\kappa,q-1}(y)
  +y^{p+q}
  \int_{E}
  \frac{\sin^{\kappa}\theta}{y-\cos\theta}d\theta.
    \end{split}
  \end{align}

  By assumption, if we take $(p,q)=(l,k)$ and $(p,q)=(0,k+l)$ respectively, correspondingly \eqref{eq27} becomes
  \begin{align*}
    &y^l\int_{E}
  \frac{\sin^{\kappa}\theta\cos^k\theta}{y-\cos\theta}d\theta
  =
   y^lT^E_{\kappa,k-1}(y)
  + y^{k+l}
  \int_{E}
  \frac{\sin^{\kappa}\theta}{y-\cos\theta}d\theta
  \end{align*}
  and
  \begin{align*}
  &\int_{E}
  \frac{\sin^{\kappa}\theta\cos^{k+l}\theta}{y-\cos\theta}d\theta
  =
   T^E_{\kappa,k+l-1}(y)
  + y^{k+l}
  \int_{E}
  \frac{\sin^{\kappa}\theta}{y-\cos\theta}d\theta,
  \end{align*}
  respectively. This implies that
  \begin{align*}
  y^l\int_{E}
  \frac{\sin^{\kappa}\theta\cos^k\theta}{y-\cos\theta}dt
  &=
  \int_{E}
  \frac{\sin^{\kappa}\theta\cos^{k+l}\theta}{y-\cos\theta}dt
  +\left(y^lT^E_{\kappa,k-1}(y)-T^E_{\kappa,k+l-1}(y)\right).
  \end{align*}
\end{proof}

 Lemma \ref{lem2.9} has the next consequence.

\begin{proposition}\label{prop6}
Let $\mathcal B_1$ and $\mathcal B_2$ be those defined in \eqref{eq44}. Then $\mathcal B_1=\mathcal B_2$.
\end{proposition}
\begin{proof} Let $T^E_{\kappa,q}$ be that defined in \eqref{eq29} with $E=E_0,\cdots,E_n$ defined in \eqref{eq25}. First we have by Lemma \ref{lem2.9} that
for $p=1,2,\cdots,[\frac{m+1}{2}]$,
  \begin{align}\label{eq33}
  y^{m-(2p-1)} C^{E}_{0,1}
  &=
   C^{E}_{m+1-2p,1}
  -T^E_{0,m-2p}
  \in\mathcal B_2.
    \end{align}
Moreover, observe that the leading coefficient of the polynomial $T^E_{\kappa,q}$ is $-\int_E\sin^{\kappa}\theta d\theta$ for any $q\in\mathbb Z_0^+$. Hence, the degree of $y^{q_1}T^E_{\kappa,q_2}(y)-T^E_{\kappa,q_1+q_2}(y)$ is at most $q_1+q_2-1$ (resp. $0$) when $q_0+q_1\geq1$ (resp. $q_0+q_1=0$).
  Thus Lemma \ref{lem2.9} also tells us that
    \begin{itemize}
\item For $i=1,\cdots,m+1$ and $p=0,\cdots,[\frac{m-i+1}{2}]$,
    \begin{align}\label{eq34}
    \begin{split}
    y^{m-(i+2p-1)} C^{E}_{i,1}
  =
  C^{E}_{m+1-2p,1}
  +\left(y^{m-(i+2p-1)}T^E_{0,i-1}-T^E_{0,m-2p}\right)
  \in
  \mathcal B_2.
  \end{split}
    \end{align}
\item For $i=0,\cdots,m$ and $p=0,\cdots,[\frac{m-i}{2}]$,
     \begin{align}\label{eq35}
     \begin{split}
  y^{m-(i+2p)}S^{E}_{i,1}
  =
   S^{E}_{m-2p,1}
  +\left(y^{m-(i+2p)}T^E_{1,i-1}-T^E_{1,m-2p-1}\right)
  \in
  \mathcal B_2.
  \end{split}
  \end{align}
  \end{itemize}
Consequently, $\mathcal B_1\subseteq\mathcal B_2$.

Now we prove that $\mathcal B_2\subseteq\mathcal B_1$. Then $\mathcal B_1=\mathcal B_2$ and the conclusion is obtained.
In fact, due to the definition of the intervals $E_0,\cdots,E_n$ in \eqref{eq25}, one can easily check that
\begin{itemize}
  \item For $p=1,2,\cdots,[\frac{m}{2}]$,
\begin{align*}
2\pi y^{m-2p}
=
\sum^{n}_{s=0}\left(
y^{m-(2p-1)}C^{E_s}_{0,1}-y^{m-2p}C^{E_s}_{1,1}
\right)
\in
\mathcal B_1,
\text{ i.e., $y^{m-2p}\in \mathcal B_1$}.
\end{align*}
  \item For $p=0,1,\cdots,[\frac{m-1}{2}]$,
\begin{align*}
&(\sin\vartheta_1-\sin\vartheta_0)y^{m-(2p+1)}
=
y^{m-2p}C^{E_0}_{1,1}-y^{m-(2p+1)}C^{E_0}_{2,1}
\in
\mathcal B_1,\\
&(\cos\vartheta_0-\cos\vartheta_1)y^{m-(2p+1)}
=
y^{m-2p}S^{E_0}_{0,1}-y^{m-(2p+1)}S^{E_0}_{1,1}
\in
\mathcal B_1,
\end{align*}
which implies that $y^{m-(2p+1)}\in\mathcal B_1$ (here we have used the fact that $\sin\vartheta_1-\sin\vartheta_0=0$ and $\cos\vartheta_1-\cos\vartheta_0=0$ do not hold simultaneously).
\end{itemize}
Hence, $1,y,\cdots,y^{m-1}\in\mathcal B_1$. Taking the equalities in \eqref{eq33}, \eqref{eq34} and \eqref{eq35} into account, we can finally get that $\mathcal B_2\subseteq\mathcal B_1$.
\end{proof}

\section*{Acknowledgements}

The first author is supported by the NSF of China (No.11401255) and
the China Scholarship Council (No. 201606785007) and the Fundamental
Research Funds for the Central Universities (No. 21614325). The second author is supported by the NNSF of China (No. 11771101), the major research program of colleges and universities in Guangdong Province, PR China (No. 2017KZDX M054), and the Science and Technology Program of Guangzhou, China (No. 20180501 0001).
The third author is partially supported by NNSF of China (No. 11671254, 11871334 and 12071284) and also by Innovation Program of Shanghai Municipal Education Commission.

\end{document}